\documentclass[12pt]{article}
\usepackage{geometry}                		% See geometry.pdf to learn the layout options. There are lots.
\usepackage[parfill]{parskip}    		% Activate to begin paragraphs with an empty line rather than an indent
\usepackage{amssymb}
\usepackage{amsmath}

\usepackage{palatino}
\usepackage{textcomp}
\usepackage{amssymb}
\usepackage{graphicx}
\usepackage{epsfig}
\usepackage{epstopdf}
\usepackage{mathrsfs}
\usepackage[english]{babel}
\usepackage{amscd}
\usepackage{amsthm}
\usepackage{amsmath}
\usepackage{subfig}
\usepackage{floatrow}
\usepackage{caption}
\usepackage{fancyhdr} % This should be set AFTER setting up the page geometry
\newtheorem{thm}{Theorem}[section]
\newtheorem{lemma}[thm]{Lemma}
\newtheorem{prop}[thm]{Proposition}
\newtheorem{cor}[thm]{Corollary}
\newtheorem{defn}{Definition}

\newtheorem{prbl}[thm]{Problem}
\newtheorem{que}[thm]{Question}

\title{On the Bernstein-Hoeffding method}

\author{Christos Pelekis\thanks{Department of Computer Science, KU Leuven, Celestijnenlaan 200A, 3001, Belgium, 
email: pelekis.chr@gmail.com} \ \ \ \
Jan Ramon\thanks{Department of Computer Science, KU Leuven, Celestijnenlaan 200A, 3001, Belgium, 
email: Jan.Ramon@cs.kuleuven.be}\ \ \ \ Yuyi Wang\thanks{Department of Computer Science, KU Leuven, Celestijnenlaan 200A, 3001, Belgium, 
email: Yuyi.Wang@cs.kuleuven.be} }

\date{}							% Activate to display a given date or no date

\begin{document}
\maketitle

%\section{}
%\subsection{}

\begin{abstract} 
We show that the Bernstein-Hoeffding method can be employed to
a larger class of generalized moments. This class 
includes the exponential moments whose properties play a key role in the proof of 
a well-known inequality of Wassily Hoeffding, for sums of independent and bounded random variables whose 
mean is assumed to be known.  
As a result we can generalise and improve upon  this inequality.
We show that Hoeffding's inequality is optimal in a broader sense.  
Our approach allows to obtain "missing" factors in 
Hoeffding's inequality whose existence is motivated by the central limit theorem. The later 
result is  a rather weaker version of a theorem that is due to Michel Talagrand. Using ideas from the theory 
of Bernstein polynomials, we 
show that the Bernstein-Hoeffding method can be adapted to the case in which one has information 
on higher moments of the random variables. 
Moreover, we consider the performance of the method under additional information 
on the conditional distribution of the random variables and, finally, we show that the method
reduces to
Markov's inequality  when employed to non-negative and 
unbounded random variables. 
\end{abstract}


\section{Introduction}
\subsection{Motivation and related work}\label{motrel}

For a given real $p\in (0,1)$ let 
$\mathcal{B}(p)$ be the set of all $[0,1]$-valued random variables whose mean is equal to $p$. 
Formally

\[ \mathcal{B}(p):= \{X: 0\leq X \leq 1, \mathbb{E}[X]= p \}. \]
The main motivation behind this work is the following, well-known, problem.\\

\begin{prbl}
\label{Hoeff} Fix $n$ real numbers $p_1,\ldots,p_n\in (0,1)$ and a
real number, $t$, such that $\sum_{i=1}^{n}p_i<t< n$.  Find (or give upper bounds on)
\[   \psi(p_1,\ldots,p_n ;t)=  \sup_{\mathbf{X}} \mathbb{P}\left[\sum_{i=1}^{n}X_i\geq t\right], \]
where the supremum is taken over all random 
vectors $\mathbf{X}=(X_1,\ldots,X_n)$ of  independent random variables
with
$X_i\in\mathcal{B}(p_i)$,  for  $i\in \{1,\ldots, n\}$.
\end{prbl}

If $t\leq \sum_i p_i$, then the problem is trivial; just choose $X_i$ to be equal to $p_i$ with probability $1$.
There is a vast amount of literature that is related to Problem \ref{Hoeff}.
The interested reader is invited to take a look at the works of Bentkus 
\cite{Bentkustwo}, \cite{Bentkusone}, \cite{Bentkusthree},  
Fan et al. \cite{Fan}, From \cite{From}, From et al. \cite{Fromtwo},
Gy{\"o}rfi et al. \cite{Tusnady}, Hoeffding \cite{Hoeffdingone}, Kha et al. \cite{Nagaev}, Krafft et al. \cite{Kraft}, 
McDiarmid \cite{McDiarmid}, Pinelis \cite{Pinelisone},\cite{Pinelistwo}, Schmidt et al. \cite{Schmidt}, 
Siegel \cite{Siegelone}, Talagrand  
\cite{Talagrand}, Xia \cite{Xia} among many others. \\
Determining the function 
$\psi(p_1,\ldots,p_n ;t)$, for given $p_1,\ldots,p_n, t$, turns out to 
be a notorious problem that has been around for many years.
To our knowledge, no solution to this problem has ever been reported and most of the existing work 
focuses towards obtaining upper bounds on the function $\psi(p_1,\ldots,p_n ;t)$ that are as tight as possible.\\
Probably the first systematic approach that allows one to obtain upper bounds on large 
deviations from the expectation 
for sums of independent, bounded random variables was performed by Hoeffding in \cite{Hoeffdingone}. 
Hoeffding's approach is based on a method of Bernstein (see \cite{Hoeffdingone}, page $14$) and from now on will be 
referred to as the \emph{Bernstein-Hoeffding} method.  
The Bernstein-Hoeffding method is, briefly, the following.\\
Markov's inequality and the assumption that the random variables are independent imply that
\[ \mathbb{P}\left[\sum_{i=1}^{n}X_i\geq t\right] \leq 
e^{-ht} \prod_{i=1}^{n} \mathbb{E}[e^{hX_i}]\leq e^{-ht}\left\{\frac{1}{n} 
\sum_{i=1}^{n}\mathbb{E}[e^{hX_i}] \right\}^n, \; \text{for all}\; h>0,\]
where the last inequality comes from the arithmetic-geometric means inequality.
By exploiting the fact that the function $f(t)=e^{ht}$ is \emph{convex} one can show that 
\[\mathbb{E}[e^{hX_i}]\leq \mathbb{E}[e^{hB_i}], \] 
where $B_i$ is a Bernoulli random variable of mean $p_i$. 
Hence we conclude that 
\[ \mathbb{P}\left[\sum_{i=1}^{n}X_i\geq t\right] \leq e^{-ht} \{(1-p)+ pe^h\}^n, \; \text{for all}\; h>0, \] 
where $p=\frac{1}{n}\sum_{i=1}^{n}p_i$. If we  minimise the expression in the right hand side of 
the last inequality with respect to $h$, we find $e^h=\frac{t(1-p)}{p(n-t)}$
and hence we obtain the following celebrated result  of Hoeffding (see \cite{Hoeffdingone}, Theorem $1$). \\

\begin{thm}[Hoeffding, $1963$]
\label{folklore}
Let $p_1,\ldots,p_n$ be $n$, given, real numbers from the interval $(0,1)$.
Let also the random variables $X_1,\ldots,X_n$ be independent and such that $X_i\in \mathcal{B}(p_i)$, for each 
$i=1,\ldots,n$.
Set $p = \frac{1}{n}\sum_{i=1}^{n}\mathbb{E}[X_i]$. Then for any $t$ such that
$np < t < n$ we have
\[ \mathbb{P}\left[\sum_{i=1}^{n}X_i  \geq t \right] \leq   \inf_{h>0}\left\{ e^{-ht}\left(1-p + pe^h\right)^n\right\} . \]
Furthermore, 
\[  \inf_{h>0}\left\{ e^{-ht}\left(1-p + pe^h\right)^n\right\}  = 
\left(\frac{p(n-t)}{t(1-p)}\right)^{t} \left(\frac{(1-p)n}{n-t}\right)^{n} := H(n,p,t)  .\]
\end{thm}

The function $H(n,p,t)$ in the last expression is the so-called 
\emph{Hoeffding bound} (or \emph{Hoeffding function}) on tail probabilities for sums of 
independent, bounded random variables. Throughout this paper, we will denote by $\text{Ber}(q)$ a 
Bernoulli random variable with mean $q$ and by $\text{Bin}(n,q)$ a binomial random variable of parameters $n$ and $q$. 
If two random variables $W,Z$ have the same distribution we will write $W\sim Z$. 
We remark that the Hoeffding bound is sharp, in the sense that 
the Bernoulli random variables $\text{Ber}(p_i)$ attain the bound, i.e., 
\[ \inf_{h>0} e^{-ht}\left\{\frac{1}{n} \sum_{i=1}^{n}\mathbb{E}[e^{hB_i}] \right\}^n = H(n,p,t) ,  \]
where $B_i$ is a  $\text{Ber}(p_i)$ random variable. 
The main ideas behind this work are hidden in the fact that 
\[  \prod_i\mathbb{E}\left[e^{hB_i}\right] = \mathbb{E}\left[e^{hB}\right], \]
where $B\sim \text{Bin}(n,p)$ is a binomial random variable of parameters $n$ and 
$p =\frac{1}{n}\sum_{i=1}^{n}\mathbb{E}[B_i]$, and 
the fact that the function $f(x)=e^{hx},h>0$, is non-negative, increasing and convex.
In a subsequent section we will show that, while applying the Bernstein-Hoeffding method, 
one can replace  the exponential function $e^{hx},h>0$,  
with any function $f(\cdot)$
having the aforementioned  properties.
Let us  mention  that Hoeffding  considered the tail probability
$\mathbb{P}\left[\sum_{i=1}^{n}X_i\geq np + nt'\right]$, where 
$0<t'<1-p$, instead of the tail $\mathbb{P}\left[\sum_{i=1}^{n}X_i\geq t\right]$, where $np<t<n$, thus 
obtaining a bound that looks different from the bound of the previous theorem. The reader is 
invited to verify that the above bound is the same as the bound given by formula $(2.1)$
in \cite{Hoeffdingone}. We choose to work with the 
tail $\mathbb{P}\left[\sum_{i=1}^{n}X_i\geq t\right]$ because it fits better to our goals.
A slightly looser but more widely used version of Hoeffding's bound is 
the function $\exp\left(-2n (t/n -p)^2\right)$, which follows 
from the fact that 
$H(n,p,t)\leq \exp\left(-2n \left(\frac{t}{n} -p\right)^2\right)$ (see \cite{Hoeffdingone}, formula (2.3)).\\

There exists quite some work  dedicated to improving Hoeffding's bound. 
See for example the work of Bentkus \cite{Bentkusone}, Pinelis \cite{Pinelistwo}, 
Siegel \cite{Siegelone} and Talagrand \cite{Talagrand}, just to name a few references. 
Let us bring the reader's to attention the following two results
that are extracted from the papers of  Talagrand (\cite{Talagrand}, Theorem $1.2$)
and Bentkus (\cite{Bentkusone}, Theorem $1.2$). Talagrand's  paper focuses
on obtaining some "missing" factors 
in Hoeffding's inequality whose existence is motivated by the Central Limit Theorem  
(see \cite{Talagrand}, Section $1$). These factors are obtained by combining 
the Bernstein-Hoeffding method together with a technique (i.e. suitable change of measure) that is used in 
the proof of Cram\'er's theorem on large deviations, yielding the following.\\

\begin{thm}[Talagrand, $1995$]
\label{Talagr}
Let $p_1,\ldots,p_n$ be $n$, given, real numbers from the interval $(0,1)$.
Let also the random variables $X_1,\ldots,X_n$ be independent and such that $X_i\in \mathcal{B}(p_i)$, for each 
$i=1,\ldots,n$. Set $p=\frac{1}{n}\sum_{i=1}^{n}p_i$. Then, for some  absolute 
constant, $K$, and every real number $t$ such that  $np+ K\leq t \leq np + np(1-p)/K$, we have
\[ \mathbb{P}\left[\sum_{i=1}^{n}X_i\geq t \right]\leq 
\left\{ \theta\left(\frac{t-np}{\sqrt{np(1-p)}}\right)+ \frac{K}{\sqrt{np(1-p)}} \right\}\cdot H(n,p,t)   ,\]
where $H(n,p,t)$ is the Hoeffding bound and $\theta(\cdot)$ is a non-negative function such that 
\[ \frac{1}{\sqrt{2\pi}(1+x)} \leq \frac{2}{\sqrt{2\pi}(x+\sqrt{x^2+4})} \leq 
\theta(x)\leq \frac{4}{\sqrt{2\pi}(3x+\sqrt{x^2+8})}, \; \text{for}\; x\geq 0 . \] 
\end{thm}

See \cite{Talagrand} for a proof of this theorem and for a precise definition of the function $\theta(\cdot)$.
In other words, Talagrand's result improves upon Hoeffding's by inserting a "missing" factor 
of order $\approx \frac{1}{t}$ in the Hoeffding bound. Notice  that Talagrand's result 
holds true for $t\in [np+ K, np+ np(1-p)/K]$, for some
absolute constant $K$ whose value does \emph{not} seem to be known. Talagrand (see \cite{Talagrand}, page 692) 
mentions that one can obtain a rather small numerical value for $K$, but numerical 
computations are left to others with the talent for it. One of the purposes of this 
paper is to improve upon Hoeffding's inequality by obtaining "missing" factors with exact numerical values for the constants. \\
Part of Bentkus' paper performs comparisons between $\mathbb{P}\left[\sum_{i=1}^{n}X_i\geq t \right]$ and tails of binomial 
and Poisson random variables.
A crucial idea in the results of \cite{Bentkusone} is to compare $\mathbb{P}\left[\sum_{i=1}^{n}X_i\geq t \right]$
with means of particular functions of certain random variables. In particular, in the proof 
of Theorem $1.2$ in \cite{Bentkusone} one can find the following result.\\

\begin{thm}[Bentkus, $2004$]
\label{Ben} 
Let the random variables $X_1,\ldots,X_n$ be independent and such that $0\leq X_i\leq 1$, for each $i=1,\ldots,n$.
Set $p = \frac{1}{n}\sum_{i=1}^{n}\mathbb{E}[X_i]$. Then, for any positive real, $t$, such that
$np < t < n$, we have
\[ \mathbb{P}\left[\sum_{i=1}^{n}X_i  \geq t \right] \leq   \inf_{a<t} \frac{1}{t-a} \mathbb{E}\left[\max\{0, B-a\} \right] ,\]
where $B\sim \text{Bin}(n,p)$. Furthermore, if $t$ is additionally assumed 
to be a \emph{positive integer}, we have
\[ \mathbb{P}\left[\sum_{i=1}^{n}X_i  \geq t \right] \leq e \cdot\mathbb{P}\left[ B \geq t\right],  \]
where $e=2.718\ldots$.
\end{thm}

The quantity on the right hand side of the first inequality is estimated in \cite{Bentkusone}, Lemma $4.2$.
We will see in the forthcoming sections that first statement 
of Bentkus' result is optimal in a slightly broader sense, i.e., it is the 
best bound that can be obtained from the inequality 
\[ \mathbb{P}\left[\sum_{i=1}^{n}X_i  \geq t\right]  \leq \frac{1}{f(t)}\mathbb{E}[f(B)] ,\]
where $f$ is a non-negative, convex and increasing function. Additionally, we will improve upon the constant $e$
of the second statement.

\subsection{Main results}

In this paper we shall be interested in employing the Bernstein-Hoeffding method to a larger class of generalized  
moments. Such approaches have been already performed by Bentkus \cite{Bentkusone}, Eaton \cite{Eaton},
Pinelis \cite{Pinelisthree},\cite{Pinelistwo}. Nevertheless, we were not able to find a systematic study of 
the classes of functions that are considered in our paper. 
We now proceed by defining a class of functions that is appropriate for the Berstein-Hoeffding method.
Let us call a function $f:[0,+\infty)\rightarrow [0,+\infty)$  \emph{sub-multiplicative} if 
$f(x+y) \leq f(x)\cdot f(y)$, for all $x,y \in [0,+\infty)$.
We will denote by $\mathcal{F}_{sic}$ the set of all functions $f:[0,+\infty)\rightarrow [0,+\infty)$
that are sub-multiplicative,  increasing and convex. Examples of such functions are $e^{hx}$, for fixed $h>0$,
$e^{hx}(1+cx)$, for fixed $h,c>0$ and so on. 
Our first  result shows that the Bernstein-Hoeffding method can be adjusted to the class $\mathcal{F}_{sic}$. \\

\begin{thm}
\label{mainHoeff} Let $\mathcal{F}_{sic}$ be defined as above.
Let the random variables $X_1,\ldots,X_n$ be independent and such that $0\leq X_i\leq 1$, for each $i=1,\ldots,n$.
Set $p = \frac{1}{n}\sum_{i=1}^{n}\mathbb{E}[X_i]$. Then, for any positive real, $t$, such that
$np < t < n$, we have
\[ \mathbb{P}\left[\sum_{i=1}^{n}X_i  \geq t \right] \leq   
\inf_{f\in \mathcal{F}_{sic}}\left\{ \frac{1}{f(t)}\left\{(1-p)f(0) + pf(1)\right\}^n\right\} . \]
\end{thm}

We prove this result in Section \ref{mainHof}. 
Theorem \ref{mainHoeff} can be deduced using the same argument as Hoeffding's result.
Its proof ought to be somewhere in the literature but we were unable to 
locate a reference. We provide a proof 
for the sake of completeness. Additionally, we prove in Section \ref{mainHof} that Hoeffding's bound 
is the best bound that can be obtained using functions from the class $\mathcal{F}_{sic}$.
In Section \ref{improve} we extend Theorem \ref{mainHoeff}
to an even larger class of moments. More precisely, fix $t>0$ and let
$\mathcal{F}_{ic}(t)$ be the set consisting of all \emph{convex} functions $f:[0,\infty)\rightarrow [0,\infty)$ 
that are \emph{increasing} in the interval $[t,\infty)$. 
Examples of such functions are $|x-\epsilon|$ for fixed 
$\epsilon \leq t$, $\max (0, x-\epsilon)$ for fixed $\epsilon \leq t$,  $e^{hx}$ for $h>0$ and so on.
In Section \ref{improve}, by employing
ideas from the theory of \emph{convex orders}, we  obtain the following.\\

\begin{thm}\label{maintwo} 
Let the random variables $X_1,\ldots,X_n$ be independent and such that $0\leq X_i\leq 1$, for each $i=1,\ldots,n$.
Set $p = \frac{1}{n}\sum_{i=1}^{n}\mathbb{E}[X_i]$. Then, for any fixed real number, $t$, such that
$np < t < n$, we have
\[ \mathbb{P}\left[\sum_{i=1}^{n}X_i  \geq t \right] \leq   \inf_{f\in \mathcal{F}_{ic}(t)}\frac{1}{f(t)}\mathbb{E}[ f(B)] , \]
where $B\sim \text{Bin}(n,p)$ is a binomial random variable and $\mathcal{F}_{ic}(t)$ is the class of functions defined above.
\end{thm} 

In Section \ref{optopt} we show that the functions $f\in \mathcal{F}_{ic}(t)$ that minimise $\frac{1}{f(t)}\mathbb{E}[f(B)]$
are those used in the aforementioned result of Bentkus, i.e., Theorem \ref{Ben}.
We then choose a particular function $\phi \in \mathcal{F}_{ic}(t)$ and obtain
a version of Talagrand's result having exact numerical constants. 
More precisely, in Section \ref{missingtwo}, we prove the following improvement upon Hoeffding's inequality.\\

\begin{thm}\label{Talagrtwo} 
Let the random variables $X_1,\ldots,X_n$ be independent and such that $0\leq X_i\leq 1$, for each $i=1,\ldots,n$.
Set $p = \frac{1}{n}\sum_{i=1}^{n}\mathbb{E}[X_i]$. Let $t$ be a fixed \emph{positive integer} such that
$\frac{enp}{ep-p+1} \leq t < n$. 
Then
\[ \mathbb{P}\left[\sum_{i=1}^{n}X_i  \geq t \right] \leq \frac{1+h}{e^h}\cdot\left( H(n,p,t)- T(n,p,t;h) \right) + 
\left(1-\frac{1+h}{e^h}\right) \mathbb{P}\left[B_{n,p} =t\right], \]
where $H(n,p,t)$ is the Hoeffding function, $B_{n,p}$ is a binomial random variable of parameters $n$ and $p$, 
\[ T(n,p,t;h)= \sum_{i=0}^{t-1} e^{h(i-t)} \mathbb{P}\left[B_{n,p}=i\right] , \]
and $h$ is such that $e^h=\frac{t(1-p)}{p(n-t)}$, i.e., it is the optimal real such that 
\[ \frac{1}{e^{ht}}\mathbb{E}[e^{hB}] = \inf_{s>0}\; \frac{1}{e^{st}} \mathbb{E}[e^{sB}], \]
with $B\sim \text{Bin}(n,p)$.
\end{thm} 

Let us illustate that 
the bound of the previous result is an improvement upon Hoeffding's inequality. Indeed, notice that  
the bound of the previous Theorem is 
\[ \leq \frac{1+h}{e^h}\cdot H(n,p,t) + \left(1-\frac{1+h}{e^h}\right) \mathbb{P}\left[B_{n,p} =t\right] \]
and the later quantity is a convex combination of $H(n,p,t)$ and $\mathbb{P}\left[B_{n,p} =t\right]$. 
Now Hoeffding's Theorem \ref{folklore} implies that
\[ \mathbb{P}\left[B_{n,p} =t\right]\leq \mathbb{P}\left[B_{n,p} \geq t\right] \leq H(n,p,t) \]
and therefore the bound of the previous result is \emph{smaller} than Hoeffding's. \\
In brief, the previous result improves upon Hoeffding's by adding a ``missing'' factor that is equal 
to $\frac{1+h}{e^h}<1$. Since  $e^h=\frac{t(1-p)}{p(n-t)}$ it follows that the ``missing'' factor can be written as
\[ \frac{1+h}{e^h} = \frac{p}{1-p}\left(\frac{n}{t}-1\right) \left(1+ \ln \frac{1-p}{p(n/t -1)}\right) . \]
On the other hand, Talagrand's result provides a factor that is approximately 
\[ \frac{\sqrt{np(1-p)}}{\sqrt{2\pi} (\sqrt{np(1-p)}+t-np)} + \frac{K}{\sqrt{np(1-p)}} . \]
Is it  unclear how to compare the two factors without knowing the constant $K$. If we assume that $K\approx 0$ then
elementary, though quite tedious, calculations show that
Talagrand's bound is sharper than the bound of Theorem \ref{Talagrtwo}. 
Our bound has the advantage that it does \emph{not} involve unknown constants and that it is obtained 
using a rather simple argument.\\
Using Theorem \ref{maintwo} we can also  obtain the following, partial, improvement upon the 
second statement of Bentkus' result, i.e., Theorem \ref{Ben}. \\

\begin{thm}
\label{binbin} Let the random variables $X_1,\ldots,X_n$ be independent and such that $0\leq X_i\leq 1$, 
for each $i=1,\ldots,n$.
Set $p = \frac{1}{n}\sum_{i=1}^{n}\mathbb{E}[X_i]$. Then, for any fixed \emph{positive integer}, $t$, such that
$np < t < n$, we have
\[ \mathbb{P}\left[\sum_{i=1}^{n}X_i  \geq t \right] \leq   \frac{t-tp}{t-np}\cdot\mathbb{P}[ B\geq t] , \]
where $B\sim \text{Bin}(n,p)$ is a binomial random variable.
\end{thm}

Note that for large $t$, say $t>\frac{2np}{1+p}$, the previous result
gives an estimate for which $\frac{t-tp}{t-np}<2$. However, for values of $t$ that 
are close to $np$, the previous result provides estimates for which $\frac{t-tp}{t-np}$ can be arbitrarily large. \\
In Section \ref{BernHoeff} we generalise the Bernstein-Hoeffding method to sums of bounded, independent random variables 
for which the first $m$ moments are known. More precisely, for given real numbers $\mu_1,\ldots, \mu_m\in (0,1)$, let 
$\mathcal{B}(\mu_1,\ldots,\mu_m)$ be the set of all $[0,1]$-valued random variables whose $i$-th moment equals 
$\mu_i, i=1,\ldots, m$. Formally, 
\[ \mathcal{B}(\mu_1,\ldots,\mu_m):= 
\{X: 0\leq X\leq 1, \mathbb{E}[X]=\mu_1,\mathbb{E}[X^2]=\mu_2\ldots, \mathbb{E}[X^m]=\mu_m\} . \]

Notice that the set may be empty.
Note also that, if $\mathcal{B}(\mu_1,\ldots,\mu_m)$ is non-empty then we have $\mu_1\geq \mu_2,\geq \cdots \geq \mu_m$.
Recall the definition of the class $\mathcal{F}_{sic}$, defined above.
The main result of Section \ref{BernHoeff} is the following.\\

\begin{thm}
\label{moments} Fix positive integers, $n,m \geq 2$ and for $i=1,\ldots,n$ let $\{\mu_{ij}\}_{j=1}^{m}$ be a sequence of 
real numbers such that $1>\mu_{i1}\geq \cdots\geq \mu_{im}>0$ and for which the class $\mathcal{B}(\mu_{i1},\ldots,\mu_{im})$ 
is non-empty. Let $X_1,\ldots,X_n$ be independent random variables such that $X_i\in \mathcal{B}(\mu_{i1},\ldots,\mu_{im})$,
for $i=1,\ldots,n$, and fix $t\in [0,n]$. Then 
\[ \mathbb{P}\left[\sum_{i=1}^{n}X_i\geq t \right] \leq 
\inf_{f\in \mathcal{F}_{sic}} \;\frac{1}{f(t)} \left\{\mathbb{E}\left[ f(T_{nm})\right]\right\}^n, \]
where 
$T_{nm}$ is the random variable that takes on values in the set $\{\frac{0}{m},\frac{1}{m},\ldots,\frac{m}{m}\}$ and,
for $j=0,1,\ldots,m$, it satisfies
\[\mathbb{P}\left[T_{nm}=\frac{j}{m}\right] = \frac{1}{n}\sum_{i=1}^{n}\binom{m}{j} \mathbb{E}\left[X_i^j(1-X_i)^{m-j}\right] . \]
\end{thm}

To our knowledge, this is the first result that considers the performance of the method under additional 
information on higher moments. 
Notice that the probability distribution of the 
random variable $T_{nm}$ does \emph{not} depend on the random variables $X_1,\ldots,X_n$. 
Indeed, using the binomial formula, it is easy to see that 
\[ \mathbb{E}\left[X_i^j(1-X_i)^{m-j}\right] = \sum_{k=0}^{m-j}\binom{m-j}{k}(-1)^{m-j-k} \mu_{i,m-k} \] 
and so $T_{nm}$ is uniquely determined by the given sequences on moments $\{\mu_{ij}\}_{i,j}$. 
We will refer to the random variable that takes values on the set $\{0,\frac{1}{m},\ldots,\frac{m}{m}\}$ with probability
$\binom{m}{j} \mathbb{E}\left[X_i^j(1-X_i)^{m-j}\right]$, for $X_i\in \mathcal{B}(\mu_{i1},\ldots,\mu_{im})$, as a
$\mathcal{B}(\mu_{i1},\ldots,\mu_{im})$-\emph{Bernstein random variable}.  Let us also mention that 
Bernstein random variables occur in the study of the so-called \emph{Hausdorff moment problem} (see Feller \cite{Fellertwo}).
A similar result holds true for the class $\mathcal{F}_{ic}(t)$; we state this result in Section \ref{BernHoeff}
and sketch its proof. In Section \ref{refinebinbin} we perform comparisons 
between $\mathbb{P}\left[\sum_{i=1}^{n}X_i\geq t \right]$ and binomial tails that depend
on the additional information on the moments.
In Section \ref{BernHoeffgame} we study the performance of the method on a certain class 
of bounded random variables that contain additional 
information on  conditional means and/or conditional distributions. We 
find random variables that are larger, in the sense of convex order, than any random variable 
from this class and prove similar 
results as above that take into account the additional information. Our approach is based on the notion 
of mixtures of random variables. Additionally, we
construct random variables $\xi_{p,\sigma}$ that are \emph{different} from  
Bernstein random variables and  are larger, in the sense of convex order, than 
any random variable from the class $\mathcal{B}(p,\sigma^2)$, consisting 
of all random variables in $\mathcal{B}(p)$ whose variance is $\sigma^2$.
In particular, in Section \ref{BernHoeffgame} we prove the following. \\

\begin{thm}\label{xitheorem} Fix positive integer $n$ and assume that, for $i=1,\ldots,n$, we are given a pair $(p_i,\sigma_i^2)$ for which
the class $\mathcal{B}(p_i,\sigma_i^2)$ is non-empty. Let $X_1,\ldots,X_n$ be independent random variables 
such that $X_i\in  \mathcal{B}(p_i,\sigma_i^2)$, for $i=1,\ldots,n$. Set $p=\frac{1}{n}\sum_i\mathbb{E}[X_i]$ and fix 
$t$ such that $np<t<n$. Then 
\[ \mathbb{P}\left[\sum_{i=1}^{n}X_i \geq t \right] \leq \inf_{f\in \mathcal{F}_{ic}(t)} \;
\frac{1}{f(t)} \mathbb{E}\left[f\left(\sum_{i=1}^{n}\xi_{p_i,\sigma_i} \right)\right] , \]
where, for $i=1,\ldots,n$ the random variable $\xi_{p_i,\sigma_i}$ is given by Lemma \ref{xirandom}.
Furthermore, the infimum on the right hand side 
is attained by a function of the form $\frac{1}{t-\varepsilon}\max\{0,x-\varepsilon \}$, for some $\varepsilon \in [0,t)$.
\end{thm}

Finally, in Section \ref{unbounded} we show that the Bernstein-Hoeffding method 
reduces to Markov's inequality when employed to non-negative and \emph{unbounded} random variables.

\section{Sub-multiplicative order}
\label{mainHof}

This subsection is devoted to the proof of Theorem \ref{mainHoeff}. The proof will make use of the following 
elementary lemma, that is interesting on its own. \\

\begin{lemma}
\label{coupling} Fix real numbers $a,b$ such that $a\leq b$.
Let $X$ be a random variable that takes values on the interval $[a,b]$ and is such that $\mathbb{E}[X]=p$.
Let $B$ be the random variable 
that takes on the values $a$ and $b$ with probabilities 
$\frac{b-\mathbb{E}[X]}{b-a}$ and $\frac{\mathbb{E}[X]-a}{b-a}$, respectively. 
Then for any convex function, $f:[a,b]\rightarrow \mathbb{R}$, we have
\[ \mathbb{E}[f(X)] \leq \mathbb{E}[f(B)] . \] 
\end{lemma}
\begin{proof} Given $X$, we
couple the random variables by setting $B_X$ to be either equal to  $a$ with probability $\frac{b-X}{b-a}$, or
equal to $b$ with probability $\frac{X-a}{b-a}$.
It is easy to see that $\mathbb{E}[B_X|X]=X$ and so
\[\mathbb{E}[B_X]= \mathbb{E}[\mathbb{E}[B_X|X]] = \mathbb{E}[X]=p .\]
Jensen's inequality now implies that
\[ \mathbb{E}[f(X)] = \mathbb{E}[f(\mathbb{E}[B_X|X])] \leq \mathbb{E}[f(B_X|X)]=\mathbb{E}[f(B_X)] ,\]
as required.
\end{proof}

We are now ready to prove our first main result. 

\begin{proof}[Proof of Theorem \ref{mainHoeff}]  Set $S_n := \sum_{i=1}^{n}X_i$ and fix a function $f\in \mathcal{F}_{sic}$.
By Markov's inequality, independence and the assumption that $f$ is increasing and submultiplicative, we conclude that
\begin{eqnarray*} \mathbb{P}\left[S_n \geq t \right] &\leq&\mathbb{P}[f(S_n) \geq f(t)]\quad  (f \;\text{is increasing})\\
&\leq& \frac{1}{f(t)} \mathbb{E}\left[f\left(\sum_{i=1}^{n}X_i\right) \right]\; (\text{Markov's inequality}) \\
&\leq& \frac{1}{f(t)} \mathbb{E}\left[\prod_{i=1}^{n} f(X_i) \right] \; (\text{sub-multiplicativity}) \\
&\leq&  \frac{1}{f(t)} \prod_{i=1}^{n} \mathbb{E}[f(X_i)] \; (\text{independence}) .
\end{eqnarray*}
Since the function $f(x)$ is convex, Lemma \ref{coupling} implies that 
\[ \mathbb{E}[f(X_i)] \leq \mathbb{E}[f(B_i)], \; \text{where} \; B_i \sim \text{Ber}(p_i) . \] 
Hence
\[  \mathbb{P}[S_n\geq t]\leq  \frac{1}{f(t)}\prod_{i=1}^{n} \mathbb{E}[f(B_i)] . \] 
Now the arithmetic-geometric means inequality yields
\[ \prod_{i=1}^{n} \mathbb{E}[f(B_i)] \leq \left\{ \frac{1}{n} \sum_{i=1}^{n} \mathbb{E}[f(B_i)] \right\}^n =
 \left\{ (1-p)f(0)+pf(1)\right\}^n \] 
and thus 
\[  \mathbb{P}[S_n\geq t]\leq \frac{1}{f(t)} \left\{(1-p)f(0)+pf(1)\right\}^n, \; f\in \mathcal{F}_{sic} . \]
The result follows.
\end{proof}

The first statement in Hoeffding's result  (Theorem \ref{folklore})  is obtained by 
adjusting the previous proof to the function 
$f(t)=e^{ht}, h>0$, which clearly belongs to $\mathcal{F}_{sic}$.
For $f\in \mathcal{F}_{sic}$, let 
\[ V_n(f,t):= \frac{1}{f(t)} \left\{ (1-p)f(0)+ p f(1) \right\}^n . \]
Notice that Theorem \ref{mainHoeff} suggests that there may be some space for improvements upon Hoeffding's bound, i.e., 
there may exist a function $\phi \in \mathcal{F}_{sic}$ such that 
$V_n(\phi,t)< V_n(e^{hx},t)$ for all $h$.
We now show that this is not the case, when $t$ is an \emph{integer}.
The following result solves the problem of finding $ \inf_{f\in \mathcal{F}_{sic}} V_n(f,t)$
in case $t$ is a positive \emph{integer}. \\

\begin{prop}
\label{integer} Let $t$ be a positive integer.
Suppose that $f\in \mathcal{F}$ is such that $V_n(f,t)\leq V_n(f',t)$ for all 
$f' \in \mathcal{F}$. Then there exists $g\in \mathcal{F}$ such that 
$V_n(g,t)= V_n(f,t)$ and 
 $g(t)= e^{ht}$, for some positive constant $h$.
\end{prop}
\begin{proof} Since $f(\cdot)$ is sub-multiplicative and non-negative, it easy to see that $f(0)\geq 1$.
For $x\geq 0$, set $g(x)= f(1)^x$. Then $g(0)=1$, $g(1)= f(1)$ and so
\[    \left\{ (1-p)g(0)+ p g(1) \right\}^n \leq \left\{ (1-p)f(0)+ p f(1) \right\}^n  .  \]
Since $t$ is a positive integer, it follows that $f(t)\leq  f(1)^t= g(t)$.
Hence $V_n(g,t)\leq V_{n}(f,t)$. The result follows upon setting $h=\ln f(1)$.
\end{proof}

Hence, for integer $t$, we have $\inf_f V_n(f,t) = H(n,p,t)$, where the infimum is taken over 
all functions $f\in \mathcal{F}_{sic}$ and $H(n,p,t)$ is the Hoeffding function that is defined in the introduction. 
Quoting Hoeffding (see \cite{Hoeffdingone}, page $15$), the bound 
\[\mathbb{P}\left[\sum_{i=1}^{n}X_i\geq t\right] \leq H(n,p,t)\]
is the best that
can be obtained from the inequality 
\[ \mathbb{P}\left[\sum_{i=1}^{n}X_i\geq t\right] \leq e^{-ht} \prod_{i=1}^{n} \mathbb{E}[e^{hB_i}],\; h>0, \]
where $B_i\sim \text{Ber}(\mathbb{E}[X_i])$.
This follows from the fact that $H(n,p,t)$ is obtained by minimising the expression on the right hand side of 
the above inequality with respect to $h>0$.
Proposition \ref{integer} shows that, in case $t$ is a positive integer, Hoeffding's bound is the best the can be 
obtained in a slightly broader sense, i.e., $H(n,p,t)$ is
the best bound on $\mathbb{P}\left[\sum_{i=1}^{n}X_i\geq t\right]$ that can be obtained by minimising
$\frac{1}{f(t)}\prod_{i=1}^{n} \mathbb{E}[f(B_i)]$ with respect to $f \in \mathcal{F}_{sic}$.

\section{Convex increasing order}

\subsection{Proof of Theorem \ref{maintwo}}
\label{improve}

In this section we prove Theorem \ref{maintwo} and 
 show that the Hoeffding bound can be improved using a larger class of functions, namely the 
class $\mathcal{F}_{ic}(t)$, defined in the introduction.
Once again, Theorem \ref{maintwo} implies that there may be some space for improvement upon Hoeffding's bound. 
We will employ this result and en route find a function $\phi \in \mathcal{F}_{ic}(t)$ such that 
\[\frac{1}{\phi(t)}\mathbb{E}[\phi(B)] < \inf_{h>0} e^{-ht} \mathbb{E}[e^{hB}] ,\]
where $B\sim \text{Bin}(n,p)$. Hence there is indeed space for improvement upon Hoeffding's bound.
The proof of Theorem \ref{maintwo} will require some well-known results and the following notion 
of ordering between random variables (see \cite{Shaked}). \\

\begin{defn} Let $X$ and $Y$ be two random variables such that 
\[ \mathbb{E}[f(X)]\leq \mathbb{E}[f(Y)], \; \text{for all convex functions} \; f:\mathbb{R}\rightarrow \mathbb{R}, \]
provided the expectations exist. Then $X$ is said to be smaller than $Y$ in the \emph{convex order}, denoted
$X\leq_{cx} Y$.
\end{defn}

The following two lemmas are well-known (see  Theorems $3.A.12$ and $3.A.37$ in \cite{Shaked} and  
Theorem $4$ in \cite{Hoeffdingone}). The first one shows that  convex order is closed under convolutions.\\

\begin{lemma}
\label{convexone} Let $X_1,\ldots, X_n$ be a set of independent random variables and let 
$Y_1,\ldots, Y_n$ be another set of independent random variables. If $X_i\leq_{cx} Y_i$, for $i=1,\ldots,n$, then 
\[ \sum_{i=1}^{n} X_i \leq_{cx} \sum_{i=1}^{n} Y_i . \]
\end{lemma}

The second lemma shows that a sum of independent Bernoulli random variables is dominated, in the sense of
convex order, by a certain binomial random variable.\\

\begin{lemma}
\label{convextwo} Fix $n$ real numbers $p_1,\ldots,p_n$ from $(0,1)$.
Let $B_1,\ldots,B_n$ be independent Bernoulli random variables with $B_i \sim \text{Ber}(p_i)$.
Then 
\[ \sum_{i=1}^{n} B_i \leq_{cx} B, \]
where $B\sim\text{Bin}(n,p)$ is a binomial random variable of parameters $n$ and $p:=\frac{1}{n}\sum_i p_i$.
\end{lemma}

The proof of Theorem \ref{maintwo} is basically an extension of the proof of Theorem \ref{mainHoeff}.

\begin{proof}[Proof of Theorem \ref{maintwo}] 
Fix $f\in \mathcal{F}_{ic}(t)$. Since $f(\cdot)$ is non-negative and increasing in  
$[t,\infty)$, Markov's inequality implies that
\[ \mathbb{P}\left[ \sum_{i=1}^{n}X_i \geq t\right] \leq 
\frac{1}{f(t)} \mathbb{E}\left[ f\left(\sum_{i=1}^{n}X_i\right) \right]. \]
Since $f(\cdot)$ is convex, Lemmata \ref{coupling} and \ref{convexone}  imply that 
\[ \mathbb{E}\left[ f\left(\sum_{i=1}^{n}X_i\right) \right] \leq \mathbb{E}\left[ f\left(\sum_{i=1}^{n}B_i\right) \right] , \]
where $B_i \sim \text{Ber}(\mathbb{E}[X_i]), i=1,\ldots,n$. Now Lemma \ref{convextwo} yields
\[ \mathbb{E}\left[ f\left(\sum_{i=1}^{n}B_i\right) \right] \leq \mathbb{E}\left[ f\left(B\right) \right]  \]
and the result follows. 
\end{proof}

Similar ideas as above have been employed to sums of independent 
Bernoulli random variables by Le\'on and Perron  in \cite{Leon}.
In a subsequent section we employ Theorem \ref{maintwo} and en route improve upon Hoeffding's inequality
by inserting certain "missing" factors. Before doing so, 
we state some results regarding the optimal function in the class $\mathcal{F}_{ic}(t)$.
  
\subsection{Optimal functions in $\mathcal{F}_{ic}(t)$}\label{optopt}   

Let the random variables $X_1,\ldots,X_n$ be independent and such that $0\leq X_i\leq 1$, for each $i=1,\ldots,n$.
Set $p = \frac{1}{n}\sum_{i=1}^{n}\mathbb{E}[X_i]$ and fix a real number, $t$, such that
$np < t < n$. We have shown in the previous section that, for $f\in \mathcal{F}_{ic}(t)$, we have
\[ \mathbb{P}\left[ \sum_{i=1}^{n}X_i \geq t\right] \leq \frac{1}{f(t)}\mathbb{E}\left[ f\left(B\right) \right] , \]
where $B\sim \text{Bin}(n,p)$. Set
\[ T_n(f,t) := \frac{1}{f(t)} \mathbb{E}[f(B)] = \frac{1}{f(t)} \sum_{i=0}^{n} f(i)\cdot \mathbb{P}[B=i] . \]
In this section we solve the problem of finding $\inf_f T_n(f,t)$, where the infimum is taken over all 
functions $f\in \mathcal{F}_{ic}(t)$. We show that the solution is related to Bentkus' result. 
We begin with an observation on the optimal function.\\

\begin{lemma}\label{opt} Let $\phi \in \mathcal{F}_{ic}(t)$ be a function such that 
$T_n(\phi, t) = \inf_f T_n(f,t)$, where the infimum is taken over all 
functions $f\in \mathcal{F}_{ic}(t)$. Then we may assume that 
$\phi(t)=1$.
\end{lemma}  
\begin{proof} If $\phi(t)\neq 1$, then we set $\phi_1(x)= \frac{1}{\phi(t)}\phi(x), x\geq 0$.
\end{proof}
  
Using this result we can find functions $f\in \mathcal{F}_{ic}(t)$ that minimise $T_n(f,t)$.\\

\begin{thm}
\label{Yuyipr} Let $\phi\in \mathcal{F}_{ic}(t)$ be such that  $T_n(\phi, t) = \inf_f T_n(f,t)$, 
where the infimum is taken over all 
functions $f\in \mathcal{F}_{ic}(t)$. Then  $\phi(x)$ equals 
$\max\{0, \frac{1}{t-\varepsilon}\cdot(x-\varepsilon) \}$, for some 
$\varepsilon \in [0,t)$.
\end{thm}
\begin{proof} We may assume that $\phi(t)=1$ and so $\phi$  is such that
\[ \inf_f  \sum_{i=1}^{n} f(i)\cdot \mathbb{P}[B=i] = T_n(\phi,t), \]
where the infimum is taken over the set $ \mathcal{Z}_{ic}(t)$, containing all
functions $f\in \mathcal{F}_{ic}(t)$ such that $f(t)=1$. 
Let $m_t:= \min\{n\in \mathbb{N}: t<n\}$ be the smallest positive integer that is larger than $t$.
Note that, by definition, $0<m_t -t \leq 1$.
For $x\geq 0$, define the function 
\[ \phi_{\ast}(x) := \max\left\{0, \frac{1}{t-\varepsilon}(x-\varepsilon) \right\}, \; \text{where}\; \varepsilon=
\frac{t\phi(m_t)-m_t}{\phi(m_t)-1}. \]
In other words, $\phi_{\ast}(\cdot)$ equals zero for $x< \varepsilon$ and for $x\geq\varepsilon$ it is a straight 
line starting from point $(0,\varepsilon)\in \mathbb{R}^2$ and passing through the points $(t,\phi(t))$ and
$(m_t, \phi(m_t))$. Note that $\varepsilon <t$ and that $\varepsilon \geq 0$; indeed, if $\varepsilon< 0$, then 
$\phi(0)>0$ and
the function $\phi_1(x)=\frac{\phi(t)-\phi(0)}{t}x +\phi(0)$ would be such that 
$\frac{np}{t}+ \phi(0)=\mathbb{E}[\phi_1(B)]\leq \mathbb{E}[\phi(B)]$, which implies that  
$\mathbb{E}[\phi(B)]$ is even worse than the bound obtained by Markov's
inequality, hence contradicts its optimality. 
Since the function $\phi(\cdot)$ is convex
it follows that for every integer $k$ in the interval $[0,n]$ we have
$\phi_{\ast}(k)\leq \phi(k)$ and this, in turn, implies that 
\[ T_n(\phi_{\ast},t) \leq T_n(\phi,t) ,\]
as required.
\end{proof}  
  
This yields the following.\\

\begin{cor}
\label{Yuyiproof} Let the parameters $n,p,t$ be as in Theorem \ref{maintwo}. Then for any $t\in (np,n)$ we have
\[ \inf_f \frac{1}{f(t)} \mathbb{E}[f(B)] = \inf_{a<t} \;\frac{1}{t-a} \mathbb{E}\left[\max\{0, B-a\} \right] , \]
where $B\sim\text{Bin}(n,p)$ and the infimum on the left hand side is taken over all functions $f\in \mathcal{F}_{ic}(t)$.
\end{cor}

Notice that we can write the function $\rho_{\varepsilon}(x):=\max\{0, \frac{1}{t-\varepsilon}\cdot(x-\varepsilon) \}$, 
for $\varepsilon \in [0,t)$, 
in the form $g_h(x):=\max\{0,h\cdot (x-t)+1\}$, where $h=\frac{1}{t-\varepsilon}$, and that 
this correspondence is injective. Notice also that, since $\varepsilon \geq 0$, 
we have $h\geq 1/t$.
The following question arises naturally from Corollary \ref{Yuyiproof}. \\

\begin{que}
What is the optimal $\varepsilon$ such that  
\[ \inf_{a<t} \;\frac{1}{t-a} \mathbb{E}\left[\max\{0, B-a\} \right] =  \mathbb{E}\left[\rho_{\varepsilon}(B) \right] \; ? \]
\end{que}

We remark that such an $\varepsilon$ will satisfy $\varepsilon \leq \lceil t\rceil -1$, 
where $\lceil t\rceil := \min\{k\in \mathbb{N}: t\leq k\}$. To see this notice that if $\varepsilon > \lceil t\rceil -1$, then 
$\rho_{\varepsilon}(\lceil t\rceil -1)=0$ and we may decrease $\varepsilon$, 
until it reaches the point $\lceil t\rceil -1$, without
increasing the value  $\mathbb{E}\left[\rho_{\varepsilon}(B) \right]$. 
Since $\varepsilon \leq \lceil t\rceil -1$ it follows that 
$h \leq \frac{1}{t-\lceil t\rceil +1}$.
Now, finding the optimal $\varepsilon$ is equivalent to finding the optimal $h$.
We are \emph{not} able to find this $h$. Nevertheless,
due to the following result, one can easily find $h$ using, say, a binary search algotithm. \\

\begin{prop}\label{Yuyiprtwo} Let the parameters $n,p,t$ be as in Theorem \ref{maintwo}. Let $h>0$ be such that
\[  \mathbb{E}\left[\max\{0, h\cdot(B-t)+1\} \right] = \inf_{s>0} \; \mathbb{E}\left[\max\{0, s\cdot(B-t)+1\} \right] ,\]
where $B\sim \text{Bin}(n,p)$.
Then we may assume that $h =\frac{1}{t-j}$, for some positive integer $j\in \{0,1,\ldots, \lceil t\rceil -1\}$. 
\end{prop}  
\begin{proof} 
Recall that
$h\in \left[\frac{1}{t},\frac{1}{t+1-\lceil t\rceil }\right]$. We have
\[ \mathbb{E}[g_h(B)] = \sum_{i=0}^{n}\binom{n}{i} p^i (1-p)^{n-i} \cdot g_h(i) . \]
The function $E(h):= \mathbb{E}[g_h(B)]$ is linear on the interval $\left[\frac{1}{t-j}, \frac{1}{t-j-1} \right]$, 
for every $j\in \{0,1,\ldots,\lceil t\rceil -1\}$. Hence the function $E(h)$ is continuous and piecewise linear on 
the interval $\left[\frac{1}{t}, \frac{1}{t-\lceil t\rceil +1} \right]$ and this 
implies that it attains its minimum at the endpoints of $\left[\frac{1}{t-j}, \frac{1}{t-j-1} \right]$, 
for some $j\in \{0,1,\ldots,\lceil t\rceil -1\}$. The result follows.
\end{proof}  
  
In the next section we obtain an improvement upon Hoeffding's bound.

\subsection{An improvement upon Hoeffding's bound}
\label{missingtwo}   
  
In this section we collect results that can be obtained by employing Theorem \ref{maintwo}.
We begin with the  proof of Theorem \ref{Talagrtwo}.

\begin{proof}[Proof of Theorem \ref{Talagrtwo}] 
Given $h>0$
define the function $f(x)= \max\{0, h(x-t)+1\}$, for $x\geq 0$. It is easy to see that $f\in \mathcal{F}_{ic}(t)$.
Let $m_t$ be the largest positive integer for which $f(m_t)=0$.
Using Theorem \ref{maintwo} and the inequality $e^x > 1+x$, for $x\in \mathbb{R}$, we estimate
\begin{eqnarray*} \mathbb{P}\left[\sum_{i=1}^{n}X_i\geq t \right] \leq 
\mathbb{E}[f(B)]&=& \sum_{i= m_t+1}^{n}(h(i-t)+1)\mathbb{P}[B=i] \\
&<& \sum_{i=m_t+1}^{n}e^{h(i-t)}\mathbb{P}[B=i]\\
&\leq& H(n,p,t),
\end{eqnarray*}
which shows that $\mathbb{E}[f(B)]$ is strictly smaller than Hoeffding's bound. 
Since we assume that  $t\geq \frac{epn}{ep-p+1}$ it follows that $h\geq 1$ which in turn implies, since $t$ is an integer, that 
$f(i)=0$, for all $i\in \{0,1,\ldots,t-1\}$.
Hence we can write
\begin{eqnarray*} 
H(n,p,t) -\mathbb{E}[f(B)] &=& \sum_{i=0}^{n}e^{h(i-t)}\mathbb{P}[B=i]  - \sum_{i=t+1}^{n}(h(i-t)+1)\mathbb{P}[B=i]  \\
&=&  \sum_{i=0}^{t-1}e^{h(i-t)}\mathbb{P}[B=i] \\
&+&  \sum_{i=t+1}^{n} \left(e^{h(i-t)}-(h(i-t) +1) \right)\mathbb{P}[B=i].
\end{eqnarray*}
For $i\geq t+1$, we have 
\begin{eqnarray*} 
e^{h(i-t)}-(h(i-t) +1)  &=& \left( 1- \frac{1+h(i-t)}{e^{h(i-t)}} \right) e^{h(i-t)} \\
 &\geq& \left( 1- \frac{1+h}{e^h} \right) e^{h(i-t)}
\end{eqnarray*}
which implies that
\begin{eqnarray*} H(n,p,t) -\mathbb{E}[f(B)] &\geq&  
\left( 1- \frac{1+h}{e^h} \right) H(n,p,t) +\frac{1+h}{e^h}\cdot \sum_{i=0}^{t-1}e^{h(i-t)}\mathbb{P}[B=i]\\  
&-& \left( 1- \frac{1+h}{e^h} \right)\mathbb{P}\left[B_{n,p}=t\right].   
\end{eqnarray*}
The result follows. 
\end{proof}

If $t$ is \emph{not} an integer, then one may use the previous bound with $t$ replaced by 
$\lfloor t\rfloor:= \max\{k\in\mathbb{N}: k\leq t\}$ since 
\[ \mathbb{P}\left[\sum_{i=1}^{n}X_i\geq t \right] \leq \mathbb{P}\left[\sum_{i=1}^{n}X_i\geq \lfloor t\rfloor \right] . \]
This result improves upon Hoeffding's bound by fitting a "missing" factor that is equal to $\frac{1+h}{e^h}<1$.  
Theorem \ref{maintwo} allows to perform comparisons with binomial tails.

\begin{proof}[Proof of Theorem \ref{binbin}] Let $\psi(x)=\max\{0,x-t+1\}$ so that $\psi(t-1)=0$ and $\psi(\cdot)\in \mathcal{F}_{ic}(t)$. 
Theorem \ref{maintwo} implies that 
\[  \mathbb{P}\left[\sum_{i=1}^{n}X_i  \geq t \right] \leq  \mathbb{E}[\psi(B)],  \]
where $B\sim \text{Bin}(n,p)$. Since $t$ is a positive integer, we can write
\[  \mathbb{E}[\psi(B)] = \sum_{i=t}^{n} (i-t+1) \mathbb{P}[B=i]  = \sum_{i=t}^{n}\mathbb{P}[B\geq i] . \]
Now we use the following, well-known, estimate on binomial tails (see Feller \cite{Feller}, page 151, formula $(3.4)$):
\[ \mathbb{P}[B\geq i] \leq \frac{i-ip}{i-np} \cdot\mathbb{P}[B=i], \; \text{for} \; i>np . \]
Therefore, 
\[  \mathbb{E}[\psi(B)] \leq \sum_{i=t}^{n} \frac{i-ip}{i-np} \cdot\mathbb{P}[B=i] \leq 
\frac{t-tp}{t-np} \cdot\mathbb{P}[B\geq t], \]
as required.
\end{proof}

Compare this result with the second statement of Bentkus' Theorem \ref{Ben}, from Section \ref{motrel}. 
Note that for large $t$, say $t>\frac{2np}{1+p}$, the previous result
gives an estimate for which $\frac{t-tp}{t-np}<2$.
In a subsequent section we will show an extension of this result. 

\section{The Bernstein-Hoeffding method}

\subsection{Proof of Theorem \ref{moments}}
\label{BernHoeff}

We begin this section with the proof of Theorem \ref{moments}.
The proof borrows ideas from the theory of Bernstein polynomials 
(see Phillips \cite{Phillips}, Chapter $7$). 
Recall that, for a function $f:[0,1]\rightarrow \mathbb{R}$, the \emph{Bernstein polynomial}  
corresponding to $f$ is defined as 
\[ B_m(f,x) = \sum_{j=0}^{m} \binom{m}{j} x^j(1-x)^{n-j} f\left(j/m\right) , \]
for each positive integer $m$.
The following is a folklore result regarding Bernstein polynomials.\\

\begin{lemma}
\label{Bernst} If $f:[0,1]\rightarrow [0,\infty)$ is convex, then 
\[ f(x) \leq B_m(f,x), \; \text{for all} \; x\in [0,1]. \]
If $f:[0,1]\rightarrow [0,\infty)$ is continuous, then
\[ \sup_{x\in [0,1]} |f(x)-B_{m}(f,x)| \rightarrow 0, \; \text{as}\; m\rightarrow \infty . \]
\end{lemma}
\begin{proof} See \cite{Phillips} Theorems $7.1.5$ and  $7.1.8$. We remark that the first statement is easy to prove
and the second arose from Bernstein's search for a proof of Weierstrass' theorem.  
\end{proof}

\begin{proof}[Proof of Theorem \ref{moments}] 
Let $f\in \mathcal{F}_{sic}$. Since $f$ is non-negative, increasing and sub-multiplicative, Markov's inequality implies that 
\begin{eqnarray*} \mathbb{P}\left[\sum_{i=1}^{n}X_i\geq t \right] &\leq& \frac{1}{f(t)} \mathbb{E}\left[f\left(\sum_{i=1}^{n}X_i\right)\right]\\
&\leq& \frac{1}{f(t)} \prod_{i=1}^{n} \mathbb{E}\left[f\left(X_i\right)\right]  \\
&\leq&  \frac{1}{f(t)} \left\{\frac{1}{n}\sum_{i=1}^{n} \mathbb{E}[f(X_i)]  \right\}^n ,
\end{eqnarray*}
where the last estimate comes from the arithmetic-geometric means inequality.
Since $f$ is convex and $X_i\in[0,1]$, Lemma \ref{Bernst} implies that 
\[ \mathbb{E}\left[f\left(X_i\right)\right] \leq \mathbb{E}\left[B_m\left(f,X_i\right) \right] . \]
Now note that 
\[ \mathbb{E}\left[B_m\left(f,X_i\right) \right] = 
\sum_{j=0}^{m} \binom{m}{j} \cdot\mathbb{E}\left[X_i^j(1-X_i)^{m-j} \right] \cdot f(j/m).\]
For $j=0,1\ldots,m$ let 
\[\pi_j:= \frac{1}{n}\sum_{i=1}^{n}\binom{m}{j} \cdot\mathbb{E}\left[X_i^j(1-X_i)^{m-j} \right] .\]
Notice also that 
\[ \mathbb{E}\left[X_i^j(1-X_i)^{m-j} \right] = \sum_{k=0}^{m-j}\binom{m-j}{k}(-1)^{m-j-k}\mu_{i,m-k} \]
which implies that $\mathbb{E}\left[X_i^j(1-X_i)^{m-j} \right]$ is the same for all random variables from 
the clsss $\mathcal{B}(\mu_{i,1},\ldots,\mu_{i,m})$.
It is easy to verify that $\sum_{j=0}^{m}\pi_j =1$; hence $\pi_j, j=0,1,\ldots,m$ is a 
probability distribution on $\{0,1,\ldots,m\}$. 
Now, if we define the 
random variable $T_{nm}$ that takes on the value $\frac{j}{m}$ with probability $\pi_j, j=0,1,\ldots,m$, we have 
\[  \frac{1}{f(t)} \left\{\frac{1}{n}\sum_{i=1}^{n} \mathbb{E}[B_m(f,X_i)]  \right\}^n = 
\frac{1}{f(t)} \left\{\mathbb{E}\left[ f(T_{nm})\right]\right\}^n ,\]
as required.
\end{proof}

Note that for $m=2$ the previous result reduces to Theorem \ref{mainHoeff}, from Section \ref{mainHof}. In particular, we
conclude the following generalisation of Hoeffding's result.\\

\begin{cor} With the same assumptions as in Theorem \ref{moments}, we have
\[ \mathbb{P}\left[\sum_{i=1}^{n}X_i\geq t \right]  \leq \inf_{h>0} e^{-ht} \sum_{j=0}^{m} \pi_j e^{h \frac{j}{m}} , \]
where $\pi_j:= \frac{1}{n}\sum_{i=1}^{n}\binom{m}{j} \cdot\mathbb{E}\left[X_i^j(1-X_i)^{m-j} \right]$.
\end{cor}

Since $B_m(f,\cdot)$ converges uniformly to $f(\cdot)$, as $m\rightarrow \infty$, we conclude that 
$\mathbb{E}[f(T_{nm})]$ can be arbitrarily close to $\frac{1}{n}\sum_{i=1}^{n} \mathbb{E}[f(X_i)]$, provided
that $m$ is sufficiently large.   
Recall the definition of the class $\mathcal{F}_{ic}(t)$ from the introduction. \\

\begin{thm}
\label{momopt}
Fix positive integers, $n,m \geq 2$ and for $i=1,\ldots,n$ let $\{\mu_{ij}\}_{j=1}^{m}$ be a sequence of 
reals such that $1>\mu_{i1}\geq \cdots\geq \mu_{im}>0$ and for which the class $\mathcal{B}(\mu_{i1},\ldots,\mu_{im})$ 
is non-empty. Let $X_1,\ldots,X_n$ be independent random variables such that $X_i\in \mathcal{B}(\mu_{i1},\ldots,\mu_{im})$,
for $i=1,\ldots,n$, and fix $t\in [0,n]$. Then 
\[ \mathbb{P}\left[\sum_{i=1}^{n}X_i\geq t \right]  \leq 
\inf_{f\in \mathcal{F}_{ic}(t)} \;\frac{1}{f(t)} \mathbb{E}\left[ f(Z_{nm})\right], \]
where 
$Z_{nm}= \sum_{i=1}^{n}Z_{i}$ is an independent sum of random variables $Z_i$ such that 
\[ \mathbb{P}[Z_{i}= j/m] = \binom{m}{j} \cdot \mathbb{E}\left[X_i^j (1-X_i)^{m-j}\right], \;\text{for}\; j=0,1,\ldots,m . \]
Moreover, 
\[ \inf_{f\in \mathcal{F}_{ic}(t)} \;\frac{1}{f(t)} \mathbb{E}\left[ f(Z_{nm})\right] = 
\inf_{a<t} \; \frac{1}{t-a}\mathbb{E}[\max\{0,Z_{nm}-a\}]. \]
\end{thm} 
\begin{proof} The argument proceeds along the same lines as the proofs of the results in 
Section \ref{improve} and Section \ref{optopt} 
and so we only sketch it. 
Part of the proof of Theorem \ref{moments} yields $X_i \leq _{cx} Z_i$, i.e., 
$\mathbb{E}[f(X_i)]\leq \mathbb{E}[B_m(f,X_i)]=\mathbb{E}[f(Z_i)]$, for convex $f$. Since the convex order is 
closed under convolutions, the first statement follows. The proof of the second statement is almost identical to 
the proof of Theorem \ref{Yuyipr}.
\end{proof}

In the previous result we found a random variable $Z_i$ such that $X_i\leq_{cx} Z_i$, 
for every $X_i\in \mathcal{B}(\mu_{i1},\ldots,\mu_{im})$.
Note that  $\mathbb{E}[Z_i]=\mathbb{E}[X_i]$, for all $i=1,\ldots,n$. However,  higher moments of $Z_{i}$ are not equal to 
the higher moments of $X_i$. 
Let us illustrate this by assuming from now on that $m=2$. 
Then $\mathbb{E}[Z_i^2]= \frac{1}{2}\mathbb{E}[X_i]+\frac{1}{2}\mathbb{E}[X_i^2] \geq \mathbb{E}[X_i^2]$ 
and so $Z_i$ may \emph{not} belong to $\mathcal{B}(\mu_{i1},\mu_{i2})$. Notice that this is not 
the case when $m=1$; i.e., when we consider random variables $X_i \in \mathcal{B}(\mu_{i1})$. In this case 
(see Theorem \ref{maintwo}) we were able to find Bernoulli random variables 
$B_i$ from the class $\mathcal{B}(\mu_{i1})$ such that 
$\mathbb{E}\left[f(X_i)\right]\leq \mathbb{E}\left[f(B_i)\right]$, for all functions $f\in \mathcal{F}_{ic}(t)$.
The following question arises naturally from the above. \\

\begin{que}Fix $\mu_1,\mu_2 \in (0,1)$ such that the 
class $\mathcal{B}(\mu_1,\mu_2)$ is non-empty. Does there exist random variable $\xi\in \mathcal{B}(\mu_1,\mu_2)$
such that $\mathbb{E}\left[f(X)\right]\leq \mathbb{E}\left[f(\xi)\right]$, for all $X\in \mathcal{B}(\mu_1,\mu_2)$ and all 
increasing and convex functions $f:[0,\infty)\rightarrow [0,\infty)$?
\end{que}

It turns out that the answer to the question is \emph{no}.  In order to convince the reader we will use
Lemma \ref{cohhen} below, taken from Cohen et al. \cite{Cohen}. Let us first fix some notation.
If $X\in \mathcal{B}(\mu_1,\mu_2)$, let $\sigma^2 := \mu_2- \mu_{1}^{2}$ be its variance.
Set $\lambda = \mu_1 - \frac{\sigma^2}{1-\mu_1}$ and
and let $C$ be the random variable that takes on the 
values $\lambda$ and $1$ with probability $\frac{1-\mu_1}{1-\lambda}$ and $\frac{\mu_1-\lambda}{1-\lambda}$, 
respectively. It is easy to 
verify that $C$ has mean $\mu_1$ and variance $\sigma^2$. 
The following result is proven in Cohen et al. \cite{Cohen} and implies that $C$ 
has the maximum moments of any order, among all random variables in $\mathcal{B}(\mu_1,\mu_2)$.\\

\begin{lemma}\label{cohhen} Let $X\in \mathcal{B}(\mu_1,\mu_2)$ and let $C$ be the random variable defined above. 
Then $\mathbb{E}\left[X^k\right] \leq \mathbb{E}\left[C^{k}\right]$, for every 
non-negative integer $k$.
\end{lemma}
\begin{proof} See \cite{Cohen}, Lemma $1.4.1$.
\end{proof}

The following is an immediate consequence of the previous lemma.\\

\begin{cor} Let $X\in \mathcal{B}(\mu_1,\mu_2)$. If $X$ is not the random variable $C$ of Lemma \ref{cohhen}, then
$\mathbb{E}\left[e^{hX}\right] < \mathbb{E}\left[e^{hC}\right]$, 
for any $h>0$.
\end{cor}
\begin{proof} Note that the inequality in the conclusion is strict.
The previous lemma implies that $\mathbb{E}\left[X^k\right] \leq \mathbb{E}\left[C^{k}\right]$, for every 
non-negative integer $k$. 
Since $X$ is not equal to $C$, there is at least one $k_0$ such that 
$\mathbb{E}\left[X^{k_0}\right] \leq \mathbb{E}\left[C^{k_0}\right]$; 
this follows from the fact that the sequence of moments uniquely determines 
that random variable (see Feller \cite{Fellertwo}, Chapter VII.3). 
Therefore, Taylor expansion implies that $\mathbb{E}\left[e^{hX}\right] < \mathbb{E}\left[e^{hC}\right]$.
\end{proof}

The following result implies that the previous question has a negative answer. \\

\begin{prop}\label{impossible}  Let $\mu_1,\mu_2\in (0,1)$ be such that $\mu_1>\mu_2$ and
set $\sigma^2 = \mu_2-\mu_1^2$.
There does \emph{not} exist random variable 
$\xi\in \mathcal{B}(\mu_1,\mu_2)$ such that 
\[ \mathbb{E}\left[f(X)\right] \leq \mathbb{E}\left[f(\xi)\right], \]
for all $X\in \mathcal{B}(\mu_1,\mu_2)$ and every $f\in \mathcal{F}_{ic}(t)$.
\end{prop}
\begin{proof} We argue by contradiction. Suppose that such a $\xi$ does exist. The previous Corollary implies that 
$\xi$ must be the random variable $C$, from Lemma \ref{cohhen}. 
We now define a random variable $C'\in \mathcal{B}(\mu_1,\mu_2)$ as follows.
Let $C'$ be such that 
\[\mathbb{P}[C'=0]= \frac{\sigma^2}{\mu_{1}^{2}+\sigma^2}\; 
\text{and}\; \mathbb{P}\left[C'=\frac{\mu_{1}^{2}+\sigma^2)}{\mu_1}\right]= \frac{\mu_{1}^{2}}{\mu_{1}^{2}+\sigma^2}. \]
If $\lambda = \mu_1 - \frac{\sigma^2}{1-\mu_1}$ is as in Lemma \ref{cohhen}, let us define the function 
$g(x)= \max\{0, \frac{x-\lambda}{1-\lambda}\}$, which is clearly increasing and convex.
It is easy to verify that 
\[ \mathbb{E}\left[g(C) \right] = \frac{\sigma^2}{(1-\mu_1)^2+\sigma^2}\quad \text{and}\quad 
\mathbb{E}\left[g(C') \right] =\frac{\mu_1 \sigma^2}{(\mu_1^2+\sigma^2)((1-\mu_1)^2+\sigma^2)}. \]
If we divide the last two equations we get
\[ \frac{\mathbb{E}\left[g(C) \right]}{\mathbb{E}\left[g(C') \right]} = \frac{\mu_1^2 +\sigma^2}{\mu_1} <1, \]
which contradicts the maximality of $\xi$. 
\end{proof}

In the next section we exploit the fact that the random variable $Z_i$, from Theorem \ref{momopt}, 
is stochastically smaller than 
a particular binomial random variable. 
Using this result, we will obtain a refined version of Theorem \ref{binbin}. 

\subsection{A refinement of Theorem \ref{binbin}}
\label{refinebinbin} 

We begin this section by recalling the following, well-known, result of Hoeffding (see \cite{Hoeffdingtwo}, Theorem $4$).\\

\begin{thm}[Hoeffding, 1956]
\label{lemhof} Fix a positive integer $s$ and let $q_1,\ldots,q_s$ be real numbers from the interval  $(0,1)$. Let 
$B_1,\ldots,B_s$ be independent Bernoulli trials with parameters $q_1,\ldots,q_s$, respectively. Then 
\[ \mathbb{P}\left[\sum_{i=1}^{s}B_i \geq c \right] \leq \mathbb{P}\left[B(s,q) \geq c\right], \; \text{when}\; c\geq sq, \]
where $q = \frac{1}{s} \sum_{i=1}^{s} q_i$ and $B(s,q)\sim \text{Bin}(s,q)$.
\end{thm}

Recall  that a random variable $W$ is \emph{stochastically smaller} than a random variable $V$, if 
$\mathbb{P}[W\geq t]\leq\mathbb{P}[V\geq t]$, for all $t$. 
Denote this by $W\leq_{st} V$. It is well known, and not so difficult  to prove, 
(see \cite{Shaked}) that $W\leq_{st} V$ if and only if
$\mathbb{E}[f(W)]\leq  \mathbb{E}[f(V)]$, for every increasing function, $f$, for which the expectations exist. 
Moreover, the stochastic order is closed under convolutions.
The following result can be found in Misra et al. \cite{Misra}.\\

\begin{thm}[Misra, Singh, Harner, 2003]
\label{stochord} Fix $m\geq 2$ and real numbers $\mu_1\geq \cdots \geq \mu_m$ from the interval $(0,1)$. 
Suppose that $X$ is a random variable from $\mathcal{B}(\mu_1,\ldots,\mu_m)$. Let
$Z$ be the random variable that takes values on the set $\{\frac{0}{m},\frac{1}{m}, \ldots,\frac{m}{m}\}$ with probabilities 
\[ \mathbb{P}[Z=j/m] = \binom{m}{j} \mathbb{E}\left[X^j(1-X)^{m-j}\right], \; \text{for}\; j=0,1,\ldots,m . \]
Then $Z$ is stochastically smaller than the random variable, $\Xi$, that takes values on the set 
$\{\frac{0}{m},\frac{1}{m}, \ldots,\frac{m}{m}\}$
with probabilities
\[ \mathbb{P}[\Xi=j/m] = \binom{m}{j} \mathbb{E}\left[X^m\right]^{j/m} \left(1-\mathbb{E}\left[X^m\right]^{1/m}\right)^{m-j}, \; 
\text{for}\; j=0,1,\ldots,m  . \]
\end{thm}
\begin{proof} See \cite{Misra}, Theorem $4.1$.
\end{proof}

Notice that the random variable $\Xi$ is such that $m\cdot\Xi$ has the distribution of a
$\text{Bin}\left(m,\mathbb{E}[X^m]^{1/m}\right)$ 
random variable.
The following result is an analogue of Theorem \ref{binbin} that takes into account the additional information on the 
moments. \\

\begin{thm} 
Fix  positive integers, $n,m \geq 2$. For $i=1,\ldots,n$ let $\{\mu_{i1},\ldots,\mu_{im}\}$ be an  $m$-tuple of 
real numbers such that $1>\mu_{i1}\geq\cdots\geq  \mu_{im}>0$ and for which the class $\mathcal{B}(\mu_{i1},\ldots,\mu_{im})$ 
is non-empty. Let $X_1,\ldots,X_n$ be independent random variables such that $X_i\in \mathcal{B}(\mu_{i1},\ldots,\mu_{im})$,
for $i=1,\ldots,n$. For $j=1,\ldots,m$ 
set $q_j:= \frac{1}{n}\sum_{i=1}^{n}\mathbb{E}\left[X_i^j \right]^{1/j}$.
Fix a \emph{positive integer} $t$  such that $nq_1 < t<n$. For $j=1,\ldots,m-1$ let $I_j$ be the interval $(nq_j+1, nq_{j+1}+1]$
and let $I_m$ be the interval $(nq_m+1, n)$. If $t\in I_j$, for some $j=1,\ldots,m$,
then 
\[ \mathbb{P}\left[\sum_{i=1}^{n}X_i\geq t \right]  \leq \min_{1\leq s\leq j}\;\left\{\frac{(st-s+1)(1-q_s)}{s(st-s+1-nq_s)}\cdot
\mathbb{P}[Bi(ns,q_s)\geq st-s+1]\right\} ,   \]
where $Bi(ns,q_s)\sim\text{Bin}\left(ns, q_s\right)$, for $s=1,\ldots,j$.
\end{thm}
\begin{proof} Note that $\{\mathbb{E}\left[X_i^j\right]^{1/j}\}_{j=1}^{m}$ is an increasing sequence, for all $i=1,\ldots,n$.
Fix $j\in\{1,\ldots,m\}$ and let $s$ be such that $1\leq s\leq j$.
Define the function $f(x)= \max\{0,x-t+1\}, x\geq 0$ and note that $f(\cdot)\in \mathcal{F}_{ic}(t)$.
Since $X_i\in \mathcal{B}(\mu_1,\ldots, \mu_m)\subseteq \mathcal{B}(\mu_1,\ldots, \mu_s)$, for all $i=1\ldots,n$,  
Theorem \ref{momopt} implies that
\[ \mathbb{P}\left[\sum_{i=1}^{n}X_i\geq t \right]  \leq   \mathbb{E}\left[\max\{0, Z_{ns}-t+1\}\right] ,  \]
where $Z_{ns}=\sum_{i=1}^{n}Z_i$ and each $Z_i$ takes that value $\frac{\ell}{s}$, for $\ell=0,1,\ldots,s$, with probability
\[   \mathbb{P}[Z_i=\ell/s] =\binom{s}{\ell} \mathbb{E}\left[X_i^{\ell}(1-X_i)^{s-\ell}\right]. \] 
From Theorem \ref{stochord} we know that each $Z_i$ is stochastically smaller than $\Xi_i$, where 
$\Xi_i$ is such that $s\cdot\Xi_i\sim \text{Bin}\left(s,\mathbb{E}[X_i^s]^{1/s}\right)$, for $i=1,\ldots,n$. Since $f(\cdot)$
is an increasing function, and the stochastic order is closed under convolutions, we conclude that 
\[  \mathbb{E}\left[\max\{0, Z_{ns}-t+1\}\right] \leq  \mathbb{E}\left[\max\{0, \Xi_{ns}-t+1\}\right],  \]
where $\Xi_{ns}= \sum_{i=1}^{n}\Xi_i$ is the independent sum of $\Xi_i$'s.  Now $\Xi_i\sim \frac{1}{s}\cdot B_i$, where
$B_i \sim  \text{Bin}\left(s,\mathbb{E}[X_i^s]^{1/s}\right)$, and so
\[  \mathbb{E}\left[\max\{0, \Xi_{ns}-t+1\}\right] =   \frac{1}{s}\mathbb{E}\left[\max\{0, B_{ns}-st+s\}\right], \]
where $B_{ns}= \sum_{i=1}^{n}B_i$.
Since $t$ is assumed to be an integer, we can write
\begin{eqnarray*} \mathbb{E}\left[\max\{0, B_{ns}-st+s\}\right] &=& \sum_{k=st-s+1}^{sn}(k -st+s)\cdot \mathbb{P}[B_{ns}=k] \\
&=& \sum_{k=st-s+1}^{n} \mathbb{P}[B_{ns}\geq k].
\end{eqnarray*}
 Since $t>nq_s +1$, Hoeffding's Theorem \ref{lemhof} implies that
\[ \mathbb{P}\left[B_{ns}\geq k \right] \leq \mathbb{P}\left[Bi(ns,q_s) \geq k \right], \; \text{for}\; k\geq nsq_s , \]
where $Bi(ns,q_s)\sim \text{Bin}\left(ns, q_s\right)$. Summarising, we have shown 
\[ \mathbb{P}\left[\sum_{i=1}^{n}X_i\geq t \right] \leq \frac{1}{s}\cdot\sum_{k=st-s+1}^{ns} \mathbb{P}[Bi(ns,q_s)\geq k] . \]
Finally, we use the following estimate on binomial tails (see Feller \cite{Feller}, page 151, formula $(3.4)$):
\[ \mathbb{P}[Bi(ns,q_s)\geq k] \leq \frac{k-kq_s}{k-nq_s} \cdot\mathbb{P}[Bi(ns,q_s)\geq k], \; \text{for} \; k>nsq_s , \]
and the result follows.
\end{proof}

In the next section we show that the Bernstein-Hoeffding method can be adapted to the case in which one 
has information on the conditional means of the random variables.

\section{Mixtures}
\label{BernHoeffgame}

\subsection{Convex orders and mixtures}
In this section we will work with independent, bounded random variables for which  we have information 
on their conditional distribution.
Let us be more precise after fixing some notation and definitions. 
Let  $m\geq 2$ be a positive integer and let $0=r_0 < r_1< \cdots < r_{m-1} < r_m =1$ be 
real numbers forming the partition $\{I_j\}_j$ of the interval $[0,1]$; where, for $j=1,\ldots,m-1$, 
we set $I_j$ to be the interval 
$[r_{j-1}, r_j)$ and  $I_m=[r_{m-1}, r_m]$. Now let $\mathcal{B}(p, \{I_j,\mu_j\})$ be the class consisting of all 
random variables $X\in  \mathcal{B}(p)$ for which $\mathbb{E}\left[X| X\in I_j\right]= \mu_j$. Formally, 
\[ \mathcal{B}(p, \{I_j,\mu_j\}) := 
\{X\in \mathcal{B}(p): \mathbb{E}\left[X| X\in I_j\right]= \mu_j, \; \text{for}\; j=1,\ldots,m\} .\]
Finally, let $\mathcal{C}(p,\{I_j, q_j\})$ be the class consisting of all random variables in 
$X\in  \mathcal{B}(p)$ for which $\mathbb{P}\left[X\in I_j\right]= q_j$, i.e., 
\[ \mathcal{C}(p,\{I_j, q_j\}):= \{X\in \mathcal{B}(p): \mathbb{P}\left[X\in I_j\right]= q_j, \; \text{for}\; j=1,\ldots,m\} .\]

Suppose that we have independent, bounded random variables for which we know  whether they belong to 
one of the above classes of random variables. In this section we show how to 
employ the Bernstein-Hoeffding method in order to obtain bounds that take the additional information into account. 
In order to do so, we will need the notion of mixture of random variables.
Recall that a \emph{mixture} of the random variables $\{Y_i\}_{i\in I}$ is defined as a random selection of one of the $Y_i$ 
according to a probability distribution on the index set $I$. The next result is a mixture-analogue of Lemma \ref{coupling}. \\

\begin{lemma}\label{mix} Let $f:[0,1]\rightarrow\mathbb{R}$ be a convex function.  Fix positive integer $m\geq 2$
and real numbers $0=r_0 < r_1< \cdots < r_{m-1} < r_m =1$.
For $j=1,\ldots,m-1$, let $I_j$ be the interval $[r_{j-1}, r_j)$ and let $I_m=[r_{m-1}, r_m]$.
If $X$ is a random variable in $\mathcal{B}(p)$, 
then there exists a random variable $\xi_X$ whose support is the set $\{r_0,r_1,\ldots,r_m\}$ such that 
$\mathbb{E}[\xi_X]=p$ and
$\mathbb{E}[f(X)]\leq \mathbb{E}[f(\xi_X)]$.
\end{lemma}
\begin{proof} Let $A_j$ be the event $\{X\in I_j\},j=1,\ldots,m$.
Define $X_j$ to be the random variable whose distribution is the conditional distribution of $X$, given $A_j$.
It is easy to see that $X$ is a mixture of $\{X_j\}_j$; $X_j$ is chosen with probability $\mathbb{P}[X\in I_j]$.
Now $X_j \in I_j$ and so Lemma \ref{coupling} implies that $X_j\leq_{cx} T_j$, where $T_j$ is the 
random variable that takes that values $r_{j-1}$ and $r_j$ with probabilities 
$\pi_j:=\frac{r_j-\mathbb{E}[X_j]}{r_j-r_{j-1}}$ and 
$1-\pi_j$ respectively. The required random variable can be obtained by 
letting $\xi_X$ 
take the value $0$ with probability 
$(1-\pi_1) \mathbb{P}[X\in I_1]$, the value $1$ with probability $\pi_m\mathbb{P}[X\in I_m]$
and, for $j=1,\ldots,m-1$, the value $r_j$ with probability 
$\pi_j\mathbb{P}[X\in I_j]+(1-\pi_{j+1})\mathbb{P}[X\in I_{j+1}]$.
\end{proof}

Note that the random variable $\xi_X$ of the previous lemma depends on the conditional 
probabilities $\mathbb{P}\left[X\in I_j\right], j=1,\ldots,m$ as well as on the conditional
means $\mathbb{E}\left[X|X\in I_j\right],j=1\ldots,m$.
So, in case we know the conditional probabilities and the conditional means of the random variables, we can find 
random variables that are larger, in the sense of convex order, than any random variable having the same conditional 
probabilities and means.
Similarly, one can find a random variable that is larger, in the sense of convex order, than any random variable 
fromthe class 
$\mathcal{B}(p, \{I_j,\mu_j\})$, i.e., when we know conditional means.\\

\begin{lemma}\label{mixber} 
Let $\mathcal{B}(p, \{I_j,\mu_j\})$ be the class defined above, corresponding to a given 
partition $\{I_j\}$ of the interval $[0,1]$. 
Then there exists random variable $\xi\in  \mathcal{B}(p)$ that concentrates mass 
on the endpoints of the interval $I_1$ and $I_m$ such that $X\leq_{cx} \xi \leq_{cx} \text{Ber(p)}$, for 
all $X\in  \mathcal{B}(p, \{I_j,\mu_j\})$. $\xi$ depends on the partition $\{I_j\}$ and the conditional means $\{\mu_j\}$.
\end{lemma}
\begin{proof} Fix  a random variable 
$X\in  \mathcal{B}(p, \{I_j,\mu_j\})$. Let $X_j$  be the random 
variable whose distribution is the conditional distribution of $X$, given  the event $X\in I_j$.
Lemma \ref{mix} implies that there exist random 
variables $T_j,j=1,\ldots,m$, that concentrate mass on the endpoints, $r_{j-1},r_j$, of the intervals $I_j$ such that 
$X_j \leq_{cx} T_j$. Note that, by construction, 
$\mathbb{E}[T_j]= \mathbb{E}\left[X_j\right]= \mu_j$, for all $j=1\ldots,m$. Now let $\xi$ be a mixture 
of the random variables $T_1$ and $T_m$; we take $\xi$
to be equal to $T_1$ with probability $\frac{\mu_m-p}{\mu_m -\mu_1}$ and equal to $T_m$ with 
probability $\frac{p-\mu_1}{\mu_m-\mu_1}$. Note that $\mathbb{E}[\xi]= p$. We now show that $X\leq_{cx}\xi$.
Fix a convex function $f:[0,1]\rightarrow\mathbb{R}$ and
let $g:[0,1]\rightarrow \mathbb{R}$ be the function whose graph is the line passing through the points 
$(\mu_1, \mathbb{E}[f(T_1)])$ and $(\mu_m, \mathbb{E}[f(T_m)])$. 
Since $f$ is convex, we have $f(x)\leq g(x)$, for all $x$ from the interval $[r_1,r_{m-1}]$.
Hence 
\begin{eqnarray*} 
\mathbb{E}\left[f(X)\right] &=& \sum_{j=1}^{n}\mathbb{E}\left[f(X_j)\right]\cdot \mathbb{P}\left[X\in I_j\right]\\
&\leq& p_1 \mathbb{E}\left[f(T_1)\right] + p_m \mathbb{E}\left[f(T_m)\right] + \sum_{j=2}^{m-1}p_j \mathbb{E}\left[f(X_j)\right]\\
&\leq& p_1 g(\mu_1) + p_m g(\mu_m) + \sum_{j=2}^{m-1}p_j \mathbb{E}\left[g(X_j)\right] .
\end{eqnarray*}
Since $g(\cdot)$ is linear, we have $g(\mu_j)= g\big(\mathbb{E}[X_j]\big) = \mathbb{E}\left[g(X_j)\right], j=1,\ldots,m$ and 
so the last expression can be written as 
\[ p_1 g(\mu_1) + p_m g(\mu_m) + \sum_{j=2}^{m-1}p_j \mathbb{E}\left[g(X_j)\right]= \mathbb{E}\left[g(X)\right] . \]
By linearity of $g(\cdot)$, we have $\mathbb{E}\left[g(X)\right] = g(p)= \mathbb{E}\left[ g(\xi)\right]$. 
Summarising, we have shown 
that 
\[ \mathbb{E}\left[f(X)\right] \leq \mathbb{E}\left[ g(\xi)\right]. \]
Once again, linearity of $g(\cdot)$ implies  
\begin{eqnarray*} \mathbb{E}\left[ g(\xi)\right] &=&  \frac{\mu_m-p}{\mu_m -\mu_1} \mathbb{E}\left[g(B_1)\right] + \frac{p-\mu_1}{\mu_m -\mu_1} \mathbb{E}\left[g(B_1)\right] \\
&=& \frac{\mu_m-p}{\mu_m -\mu_1} g(\mu_1) + \frac{p-\mu_1}{\mu_m -\mu_1}  g(\mu_m) \\
&=& \frac{\mu_m-p}{\mu_m -\mu_1}  \mathbb{E}\left[f(T_1)\right]) + \frac{p-\mu_1}{\mu_m -\mu_1}  \mathbb{E}\left[f(T_m)\right]) \\
&=& \mathbb{E}\left[f(\xi) \right] 
\end{eqnarray*}
and the result follows. 
\end{proof}

The following theorem can be regarded as an improvement upon Hoeffding's in the case where one has additional information 
on the conditional means of the random variables. \\

\begin{thm} Let the random variables $X_1,\ldots,X_n$ be independent and such that $0\leq X\leq 1$.
Fix positive integer $m\geq 2$ and real numbers $0=r_0 < r_1< \cdots < r_{m-1} < r_m =1$.
For $j=1,\ldots,m-1$, let $I_j$ be the interval $[r_{j-1}, r_j)$ and let $I_m=[r_{m-1}, r_m]$.
Assume further that for $i=1,\ldots,n$ there is a sequence $\{\mu_{ij}\}_{j=1}^{m}$ 
such that $\mathbb{E}\left[X_i| X_i \in I_j\right] = \mu_{ij}$.
Let $p_i=\mathbb{E}[X_i]$.
If $t$ is such that $np<t<n$, then there exist $\pi_1,\pi_2,\pi_3,\pi_4\in (0,1)$ that add  up to $1$, such that 
\[ \mathbb{P}\left[\sum_{i=1}^{n}X_i \geq t \right] \leq 
\inf_{h>0}\; e^{-ht} \left\{\pi_1 + e^{hr_1}\pi_2 + e^{hr_{m-1}}\pi_3 + e^{hr_m}\pi_4 \right\}^n .\]
\end{thm}
\begin{proof} The proof is similar to the proof of Theorem \ref{maintwo} and so we only sketch it.
Lemma \ref{mixber} implies that $X_i \leq_{cx} \xi_i$, for $i=1,\ldots,n$, where each $\xi_i$ 
concentrates mass on the endpoints, $r_0,r_1,r_{m-1},r_m$, of the intervals $I_1$ and $I_m$.  Hence, the 
arithmetic-geometric means inequality implies
\[ \mathbb{P}\left[\sum_{i=1}^{n}X_i \geq t \right] \leq  
e^{-ht} \left\{\frac{1}{n}\sum_{i=1}^{n} \mathbb{E}\left[e^{h\xi_i}\right] \right\}^n.  \]
Set $q_i = \frac{\mu_{im}-p_i}{\mu_{im}-\mu_{i1}}$, $s_i = \frac{r_1-\mu_{i1}}{r_1-r_0}$ and 
$u_i = \frac{r_m-\mu_{im}}{r_m-r_{m-1}}$, for $i=1,\ldots,n$.
Using Lemma \ref{mix}, the result is obtained by setting 
$\pi_1 = \frac{1}{n}\sum_{i=1}^{n}q_i\cdot s_i$, 
$\pi_2 = \frac{1}{n}\sum_{i=1}^{n} q_i\cdot (1-s_i)$, $\pi_3=\frac{1}{n}\sum_{i=1}^{n}(1-q_i)\cdot u_i$ and 
$\pi_4=\frac{1}{n}\sum_{i=1}^{n}(1-q_i)\cdot (1-u_i)$. 
\end{proof}

In case one considers random variables from the class $\mathcal{C}(p,\{I_j, q_j\})$,  the random variable 
that is largest in the convex order is given by  the solution of a linear program. 
In particular, we have the following. \\

\begin{lemma} Fix a convex and increasing function $f:[0,\infty)\rightarrow [0,\infty)$.
Fix a positive integer $m\geq 2$ and real numbers $0=r_0 < r_1< \cdots < r_{m-1} < r_m =1$.
For $j=1,\ldots,m-1$, let $I_j$ be the interval $[r_{j-1}, r_j)$ and let $I_m=[r_{m-1}, r_m]$.
Assume further that there is a sequence $\{q_{j}\}_{j=1}^{m}$ 
such that the class $\mathcal{C}(p,\{I_j, q_j\})$ is non-empty. 
Then there is a $\xi \in \mathcal{B}(p)$ such that $\xi\leq_{cx} \text{Ber}(p)$ and
\[ \mathbb{E}\left[e^{hX}\right] \leq \mathbb{E}\left[e^{h\xi}\right], \; \text{for all} \; X\in  \mathcal{C}(p,\{I_j, q_j\}) ,\]
where $h$  is such that $e^h=\frac{t(1-p)}{p(n-t)}$, i.e., it is the optimal real such that 
\[ \frac{1}{e^{ht}}\mathbb{E}[e^{hB}] = \inf_{s>0}\; \frac{1}{e^{st}} \mathbb{E}[e^{sB}], \]
with $B\sim \text{Bin}(n,p)$.
The random variable $\xi$ depends on the solution of a linear program. 
\end{lemma}
\begin{proof} Since $\mathbb{E}\left[\xi\right] =p$, the first statement is evident by Lemma \ref{coupling}.
Let $X\in \mathcal{C}(p,\{I_j, q_j\})$ and set $\mu_j = \mathbb{E}\left[X|X\in I_j\right]$
From Lemma \ref{mix} we know that, for $i=1,\ldots,n$, there is a random variable $\xi_X$ that 
concentrates mass on the set $\{r_0,\ldots,r_m\}$ such that $\xi_X$ 
takes the value $r_0=0$ with probability $\pi_0:=q_1 \frac{r_1-\mu_1}{r_1}$, the value $r_m=1$ with probability 
$\pi_m:=q_m \frac{\mu_m-r_{m-1}}{r_m-r_{m-1}}$
and, for $j=1,\ldots,m-1$, the value $r_j$ with 
probability $\pi_j:=q_j \frac{\mu_j-r_{j-1}}{r_j-r_{j-1}}+q_{j+1}\frac{r_{i+1}-\mu_{i+1}}{r_{i+1}-r_i}$.
Therefore, 
\[ \mathbb{E}\left[e^{h\xi_X}\right] = \pi_0 e^{hr_0} + \pi_me^{hr_m}  + \sum_{j=1}^{m-1} \pi_j e^{hr_j} ,\]
which implies that $\mathbb{E}\left[e^{h\xi_X}\right]$ is a linear function of $\{\mu_j\}$. 
The required $\xi$ is obtained by maximising $\mathbb{E}\left[e^{h\xi_X}\right]$ subject to 
the following linear constraints: $r_{j-1} \leq \mu_j \leq \mu_j$, for $j=1,\ldots,m$ and $\sum_{j=1}^{m}q_i\mu_i = p$.
\end{proof}

\subsection{Yet another bound for cases with known variance}

Let us, for convenience,  change a bit our notation and 
set $\mathcal{B}(p,\sigma^2)$ to be the class of random variables from $\mathcal{B}(p)$ 
whose variance is $\sigma^2$. Throughout this section we will assume that $\sigma^2$ is strictly positive.
Hence $0<\sigma^2 \leq p(1-p)$. 
From Proposition \ref{impossible} we know that there does \emph{not} exist $\xi\in \mathcal{B}(p,\sigma^2)$ such that 
$X\leq_{cx} \xi$, for all $X\in \mathcal{B}(p,\sigma^2)$. From Lemma \ref{coupling} we know that 
$X\leq_{cx} \text{Ber}(p)$, for all $X\in \mathcal{B}(p,\sigma^2)$ but $\text{Ber}(p)$ does \emph{not} belong to the class 
$\mathcal{B}(p,\sigma^2)$, when $\sigma^2<p(1-p)$. In Theorem \ref{momopt} we have obtained, 
using  Bernstein polynomials, a random variable
$Z_{p,\sigma}\in \mathcal{B}(p)$, that does \emph{not} belong to $\mathcal{B}(p,\sigma^2)$, 
such that $X\leq_{cx} Z_{p,\sigma} \leq_{cx}  \text{Ber}(p)$. In this section we will 
construct another random variable having this property. More precisely, we will prove the following.  \\

\begin{lemma}\label{xirandom} 
 There exists a random variable $\xi_{p,\sigma}\in \mathcal{B}(p)$ such that 
\[ X\leq_{cx} \xi_{p,\sigma} \leq_{cx}  \text{Ber}(p) , \]
for all $X\in \mathcal{B}(p,\sigma^2)$. 
\end{lemma}

Depending on the value of $p$ and $\sigma^2$, the random variable $\xi_{p,\sigma}$
can yield efficiently computable bounds that are sharper than existing, well-known, bounds.
After stating our main results, we will  
provide some figures that illustrate the differences between the bounds.  
In order to construct $\xi_{p,\sigma}$ we will apply Lemma \ref{mix} to the partition 
$[0,p)\cup [p,1]$. We will also need the following result that is interesting on its own. \\ 

\begin{lemma}\label{nice} Suppose that $X,Y$ are two random variables from the class $\mathcal{B}(p,\sigma^2)$ and 
consider the partition  $[0,p)\cup [p,1]$ of $[0,1]$.
Let $\xi_X$ and $\xi_Y$ be the associated random variables given by Lemma \ref{mix}. Then 
\[ \xi_X \leq_{cx} \xi_Y \quad \text{if and only if} \quad \mathbb{P}[\xi_X =p] \geq \mathbb{P}[\xi_Y =p] . \]
\end{lemma}
\begin{proof} Assume first that $\xi_X \leq_{cx} \xi_Y$. Then $\mathbb{E}\big[f(\xi_X)\big] \leq \mathbb{E}\big[f(\xi_Y)\big]$ 
for the function $f:[0,1]\rightarrow [0,\infty)$ having values $f(0)=f(1)=1, f(p)=0$ and which is  linear on the intervals 
$[0,p]$ and $[p,1]$.
Hence 
\[ \mathbb{E}\big[f(\xi_X)\big] = 1- \mathbb{P}[\xi_X =p] \leq \mathbb{E}\big[f(\xi_Y)\big] = 1-  \mathbb{P}[\xi_Y =p] \]
and so $\mathbb{P}[\xi_X =p] \geq \mathbb{P}[\xi_Y =p]$. \\
Assume now that $\mathbb{P}[\xi_X =p] \geq \mathbb{P}[\xi_Y =p]$. Then for any convex function $f:[0,1]\rightarrow [0,\infty)$, 
we have $f(p)\leq (1-p)\cdot f(0)+p\cdot f(1)$ and so
\begin{eqnarray*} \mathbb{E}\big[f(\xi_X)\big]  -\mathbb{E}\big[f(\xi_Y)\big] &\leq& f(0)\cdot \left(\mathbb{P}[\xi_X =0]-\mathbb{P}[\xi_Y =0]\right) \\
&+& f(1)\cdot \left(\mathbb{P}[\xi_X =1]-\mathbb{P}[\xi_Y =1]\right) \\
&+& ((1-p)\cdot f(0)+p\cdot f(1))\cdot \left(\mathbb{P}[\xi_X =p]-\mathbb{P}[\xi_Y =p]\right) \\
&=& f(0)\cdot \left(\mathbb{E}\big[1-\xi_X\big] -\mathbb{E}\big[1-\xi_Y\big]  \right) + f(1)\cdot \left(\mathbb{E}\big[\xi_X\big] -\mathbb{E}\big[\xi_Y\big]  \right) \\
&=& 0 ,
\end{eqnarray*}
where the last equality comes form the fact that $\mathbb{E}\big[\xi_X\big] =\mathbb{E}\big[\xi_Y\big]$. 
\end{proof}

\begin{proof}[Proof of Lemma \ref{xirandom}] 
Consider the class $\mathcal{B}(p,\sigma^2)$ with $0<\sigma^2\leq p(1-p)$.
For every $X\in \mathcal{B}(p,\sigma^2)$ let $X_1$ be the random variable whose distribution 
is the conditional distribution of $X$, given 
that $X\in [0,p)$.  Let also $X_2$ be the random variable whose distribution 
is the conditional distribution of $X$, given $X\in [p,1]$.
From Lemma \ref{mix} we know that there is a random variable $\xi_X$ such that $\mathbb{E}[X]=\mathbb{E}[\xi_X]$ and 
$X\leq_{cx} \xi_X$. Furthermore, $\xi_X$ is the mixture of random variables $B_1$ and $B_2$ such that 
$\mathbb{E}[X_i]=\mathbb{E}[B_i]$ and $X_i\leq_{cx} B_i$, for $i=1,2$. 
In addition, $B_1$ concentrates mass on the set $\{0,p\}$ and $B_2$ concentrates mass on the set $\{p,1\}$.
Assume that $\xi_X$ is equal to $B_1$ with probability $\theta_X$.
Clearly, $\xi_X$ depends on $X$ and we now show how one can get rid of this dependence. 
Define $\xi_{p,\sigma}$ to be the random variable $\xi_{X}, X\in \mathcal{B}(p,\sigma^2)$ for which 
\[ \mathbb{P}\big[\xi_X=p\big] = \min_{Y\in \mathcal{B}(p,\sigma^2)}\mathbb{P}\big[\xi_Y=p\big] . \]
From Lemma \ref{nice} we have $Y\leq_{cx} \xi_{p,\sigma}$, for all $Y\in \mathcal{B}(p,\sigma^2)$.
Set $\ell_1 = p- \mathbb{E}[X_1]$ and $\ell_2 = \mathbb{E}[X_2]-p$. 
Off course, $\ell_1,\ell_2$ depend on $X$. Since 
$p=\theta_X \mathbb{E}[X_1]+ (1-\theta_X)\mathbb{E}[X_2]$, we can write
\[ \theta_X \ell_1 = (1-\theta_X)\ell_2 \; \Leftrightarrow\; \theta_X = \frac{\ell_2}{\ell_1+\ell_2} . \]
The cases in which $\ell_1$ and $\ell_2$ are both equal to zero can be excluded since they correspond 
to a constant random variable.
We now proceed to find the distribution of $\xi_{p,\sigma}= \xi_X$.  
Using Lemma \ref{mix} we compute
\begin{eqnarray*} \mathbb{P}\left[\xi_X=p\right] &=& \theta_X \mathbb{P}[B_1 =p] + (1-\theta_X)\mathbb{P}[B_2=p] \\
&=& \frac{\ell_2}{\ell_1+\ell_2} \left(1- \frac{\ell_1}{p}\right) + \frac{\ell_1}{\ell_1+\ell_2} \left(1-\frac{\ell_2}{1-p} \right) \\
&=& 1-\frac{\ell_1\ell_2}{(1-p)p(\ell_1 + \ell_2)} \\
&=& 1 - \frac{1}{(1-p)p(1/\ell_1 + 1/\ell_2)}.
\end{eqnarray*}
The last expression implies that $\mathbb{P}\left[\xi_X=p\right]$ is a decreasing function of $\ell_1$ and of $\ell_2$. 
Similarly, one can check that 
\[  \mathbb{P}\left[\xi_X=0\right] = \frac{\ell_1\ell_2}{p(\ell_1+\ell_2)} \quad \text{and}\quad  \mathbb{P}\left[\xi_X=1\right] = \frac{\ell_1\ell_2}{(1-p)(\ell_1+\ell_2)} .\]
By the Law of total variance we have 
\[ \sigma^2 = \text{Var}[X]= \theta_X \text{Var}[X_1] + (1-\theta_X)\text{Var}[X_2]+ \theta_X \ell_1^2 + (1-\theta_X)\ell_2^2. \]
Hence $ \sigma^2\geq  \theta_X \ell_1^2 + (1-\theta_X)\ell_2^2$ or, equivalently, $\sigma^2\geq \ell_1\ell_2$. 
Since $\mathbb{P}\left[\xi_X=p\right]$ is a decreasing function of $\ell_1$ and of $\ell_2$ it follows that it attains its minimum 
when $\sigma^2 = \ell_1\ell_2$ and this, in turn, implies that 
\[ \mathbb{P}\left[\xi_X=p\right] = 1-\frac{\sigma^2}{(1-p)p(\ell_1 + \ell_2)} . \]
Therefore, in order to minimise $\mathbb{P}\left[\xi_X=p\right]$ it is enough to  solve the following  
optimization problem:
\begin{align*}
 \min_{\ell_1,\ell_2}&\quad \ell_1+\ell_2\\
 \hbox{s.t.}&\quad \ell_1\ell_2=\sigma^2\\
 &\quad 0<\ell_1\leq p\\
 &\quad 0<\ell_2\leq 1-p.\\
\end{align*}
Elementary, though quite tedious, calculations show that the optimal solution $(\ell_{1}^{\ast},\ell_{2}^{\ast})$ equals
$$(\ell_{1}^{\ast},\ell_{2}^{\ast}) = \begin{cases}
   (\frac{\sigma^2}{1-p},1-p),&\qquad \text{if}\; \sigma>1-p\\
   (p,\frac{\sigma^2}{p}),& \qquad \text{if}\;\sigma>p\\
  (\sigma,\sigma), &\qquad \text{if}\; \sigma\leq \min\{p,1-p\}.\end{cases}
$$
Therefore, the required random variable $\xi_{p,\sigma}$ has the following distribution:
\begin{itemize}
 \item If $\sigma>1-p$, then $\xi_{p,\sigma}$ takes the values $0,p$ and $1$ with probability
$\frac{\sigma^2 - p \sigma^2}{p ((1 - p)^2 + \sigma^2)}$, $\frac{(1-p) ((1-p) p-\sigma^2)}{p ((p-1)^2+\sigma^2)}$ and 
$\frac{\sigma^2}{ (1 - p)^2 + \sigma^2}$ respectively.
  \item If $\sigma>p$, then $\xi_{p,\sigma}$ takes the values $0,p$ and $1$ with probability
$\frac{\sigma^2 }{p^2 + \sigma^2}$, $\frac{p ((1-p) p-\sigma^2)}{(1-p) (p^2+\sigma^2)}$ and 
$\frac{p\sigma^2}{ (1 - p)(p^2 + \sigma^2)}$, respectively. 
\item If $\sigma \leq\min\{p,1-p\}$, then $\xi_{p,\sigma}$ takes the values $0,p$ and $1$ with probability
$\frac{\sigma}{2p}$, $1-\frac{\sigma}{2(1-p)p}$ and
$\frac{\sigma}{ 2 -2 p}$, respectively. 
\end{itemize}
\end{proof}

\begin{proof}[Proof of Theorem \ref{xitheorem}]
The proof of Theorem \ref{xitheorem}  is an  application of
Lemma \ref{xirandom}. It is very similar to the proof of Theorem 
\ref{maintwo} and Theorem \ref{Yuyipr} and so we briefly sketch it. 
The first statement can be proven in the same way as Theorem \ref{maintwo}. 
The second statement follows from the fact that
$\sum_i \xi_{p_i,\sigma_i}\in \{0,p,2p,\ldots,np\}$ and by looking at the smallest 
positive integer, $m_t$, that is $>t$. As in Theorem \ref{Yuyipr}, we can find an $\varepsilon>0$ such that 
the function, $\phi(\cdot)$, that is equal to $0$ for $x<\varepsilon$ and, for $x\geq \varepsilon$, it is a straight 
line passing through the points $(t,\phi(t))$ and $(m_t,\phi(m_t))$ satisfies $\phi(t)=1$ and
\[\mathbb{E}\left[\phi\left(\sum_i \xi_{p_i,\sigma_i}\right)\right]\leq 
\mathbb{E}\left[f\left(\sum_i \xi_{p_i,\sigma_i}\right)\right], \]
for a supposedly optimal function $f(\cdot)$ with $f(t)=1$.
\end{proof}

It is not easy to find a closed form of the bound given by Theorem \ref{xitheorem}. Nevertheless, the 
bound can be easilly implemented. Note that the previous bound concerns functions from the class $\mathcal{F}_{ic}(t)$.
We end this section by performing some pictorial  comparisons between several bounds discussed in this article. 
Before doing so, let us bring to the reader's attention the following, well-known, bound that is due to Bennett \cite{Bennett}.
Bennett's approach was simplified by Cohen et al. \cite{Cohen}. In particular, by 
employing the Bernstein-Hoeffding method to the exponential function, Cohen et al. have shown the  
following. \\

\begin{thm}[Bennett bound]
Fix positive integer $n$ and assume that we are given a pair $(p,\sigma^2)$ for which
the class $\mathcal{B}(p,\sigma^2)$ is non-empty. Let $X_1,\ldots,X_n$ be independent random variables 
such that $X_i\in  \mathcal{B}(p,\sigma^2)$, for $i=1,\ldots,n$. Fix $t\in(np,n)$. Then 
\[ \mathbb{P}\left[\sum_{i=1}^{n}X_i\geq t\right] \leq \left\{\left(\frac{\alpha}{\beta}\right)^{\beta} 
\left( \frac{1-\alpha}{1-\beta}\right)^{1-\beta} \right\}^n,  \]
where $\alpha =\frac{\sigma^2}{\sigma^2+(1-p)^2}$ and $\beta =\frac{\sigma^2 + (t/n -p)(1-p)}{\sigma^2+(1-p)^2}$.
\end{thm}
\begin{proof}
 See \cite{Cohen}.
\end{proof}

\begin{figure}[htb!]
\subfloat[][$t=0.30\times20$]{\includegraphics[scale=0.45]{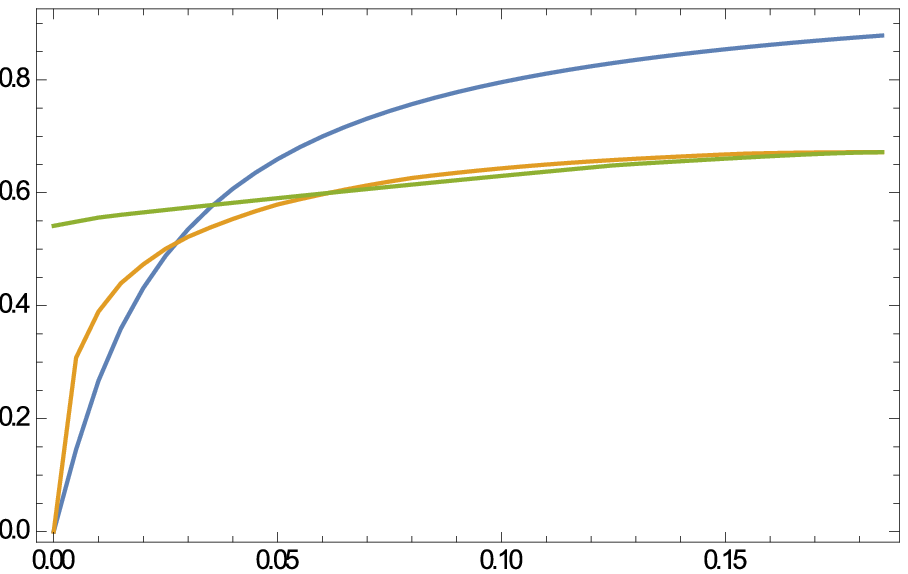}}
  \subfloat[][$t=0.55\times20$]{\includegraphics[scale=0.45]{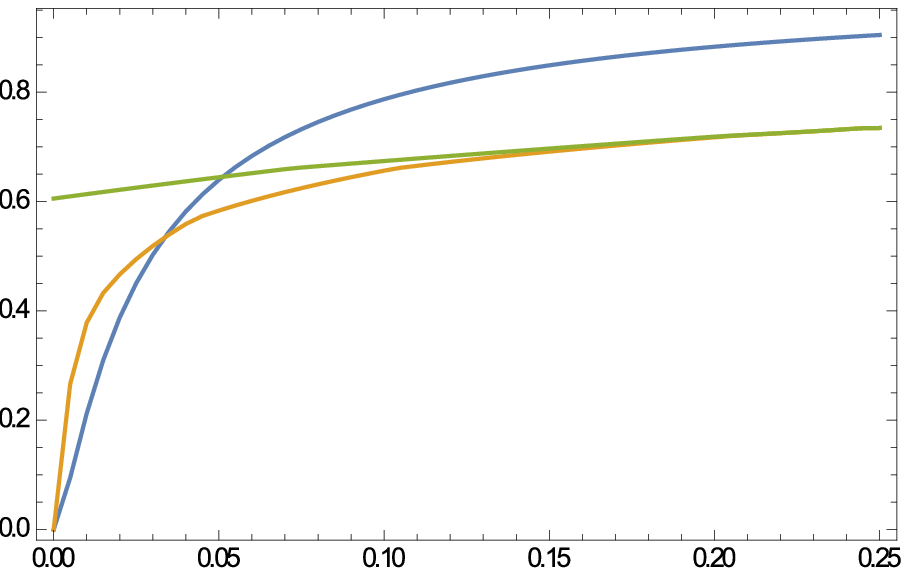}}
   \subfloat[][$t=0.80\times20$]{\includegraphics[scale=0.45]{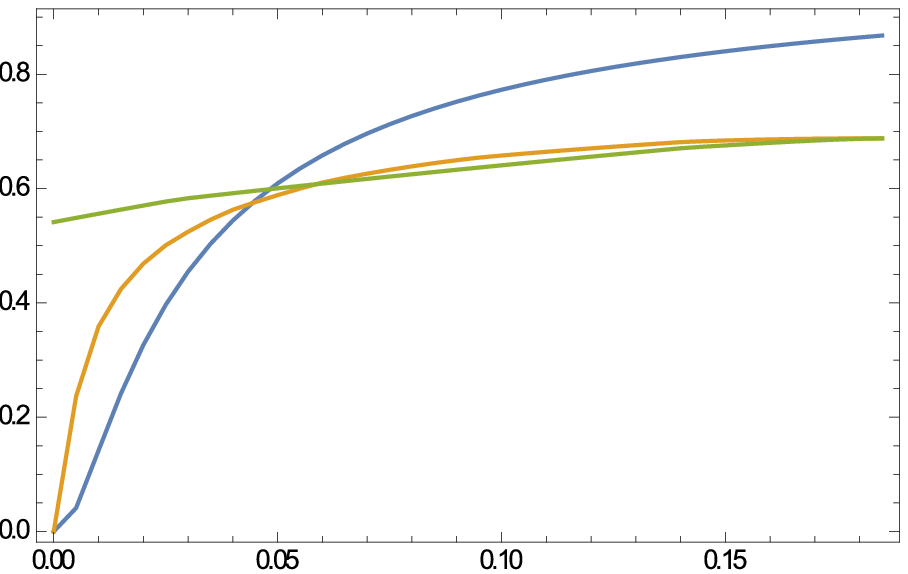}}\\
   \subfloat[][$t=0.35\times20$]{\includegraphics[scale=0.45]{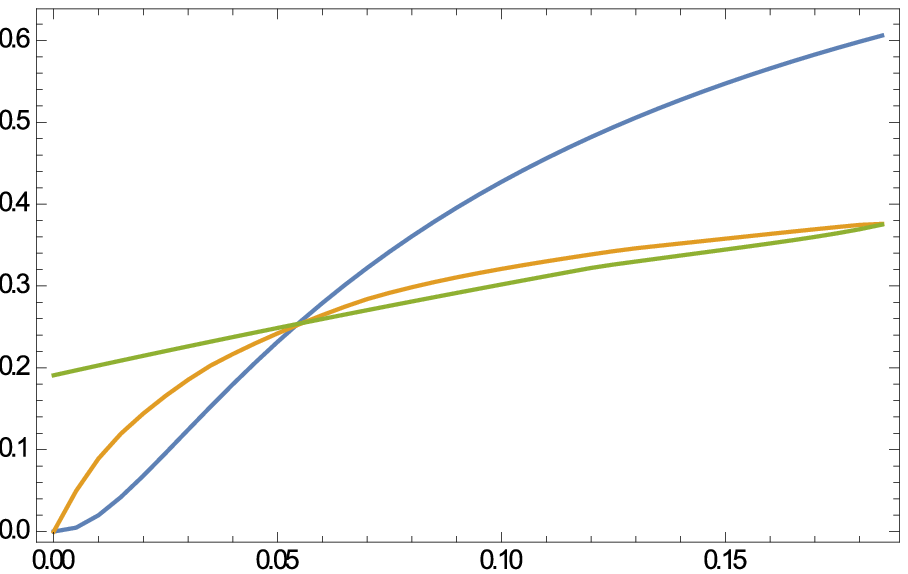}}
   \subfloat[][$t=0.60\times20$]{\includegraphics[scale=0.45]{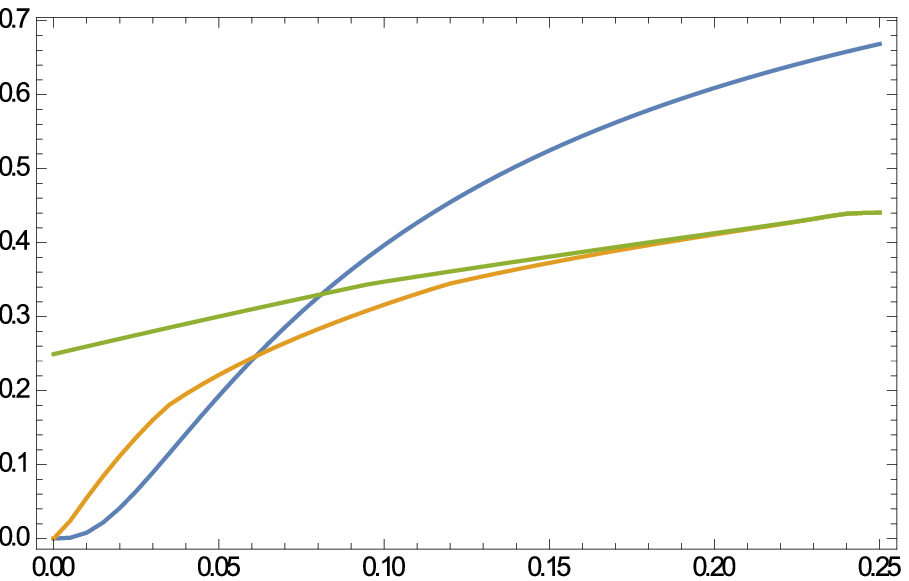}}
   \subfloat[][$t=0.85\times20$]{\includegraphics[scale=0.45]{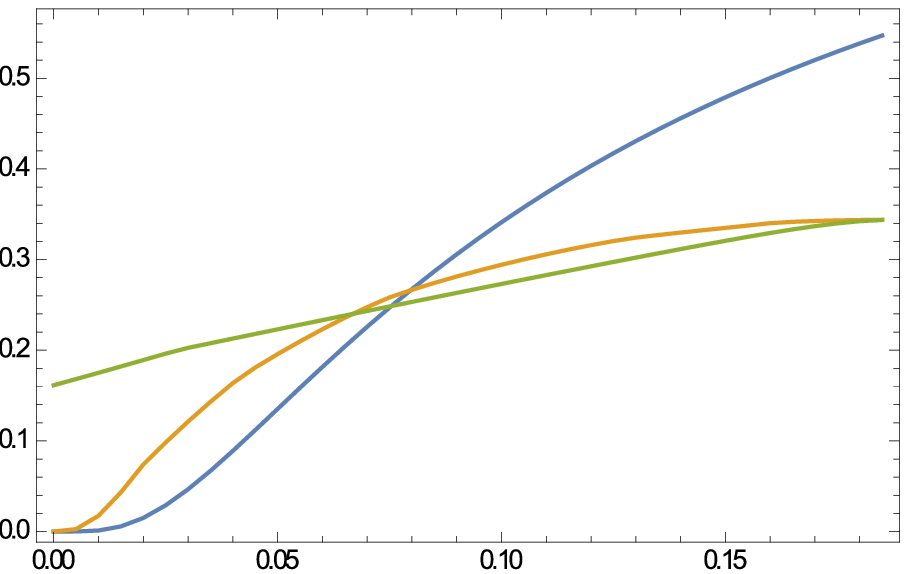}}\\
% $\epsilon=0.15$\\
  \subfloat[][$t=0.40\times20$]{\includegraphics[scale=0.45]{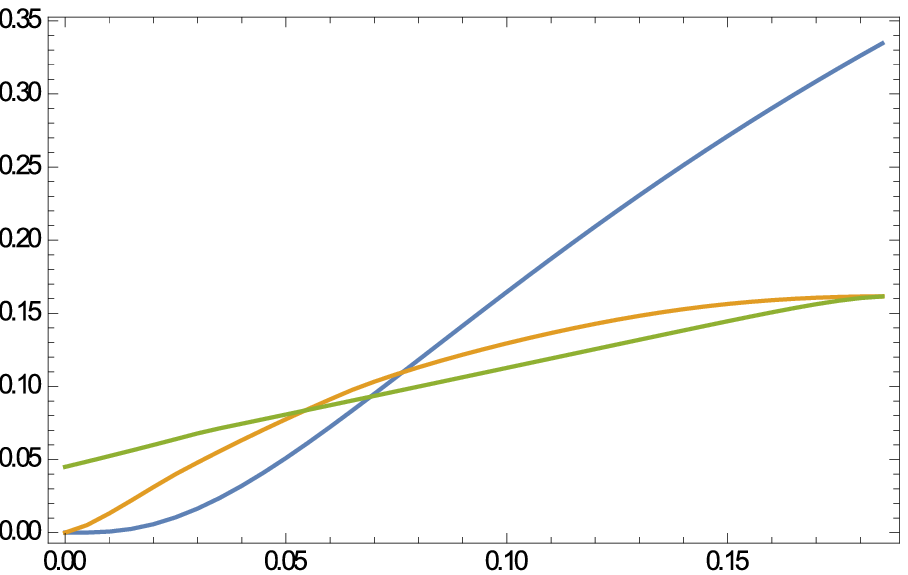}}
   \subfloat[][$t=0.65\times20$]{\includegraphics[scale=0.45]{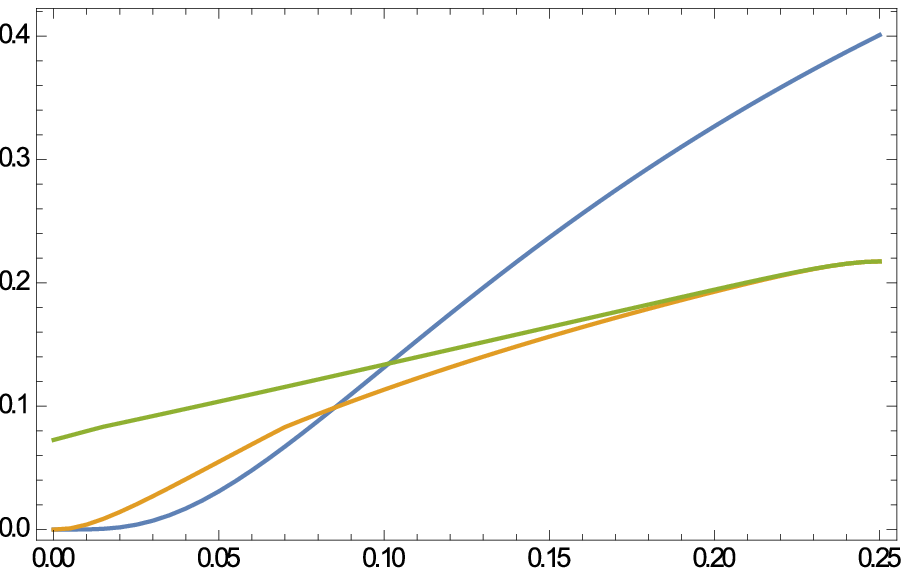}}
   \subfloat[][$t=0.90\times20$]{\includegraphics[scale=0.45]{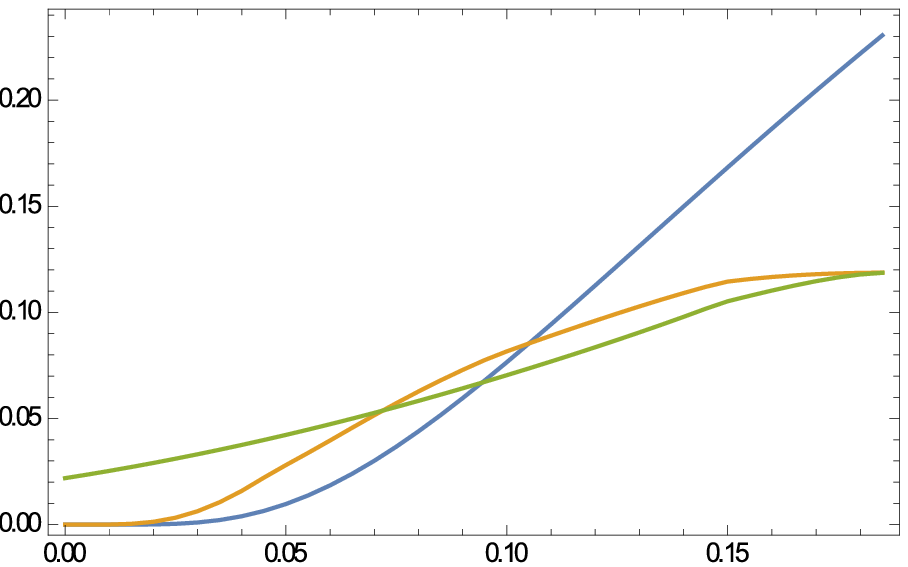}}\\
% $\epsilon=0.20$\\
  \subfloat[][$t=0.45\times20$]{\includegraphics[scale=0.45]{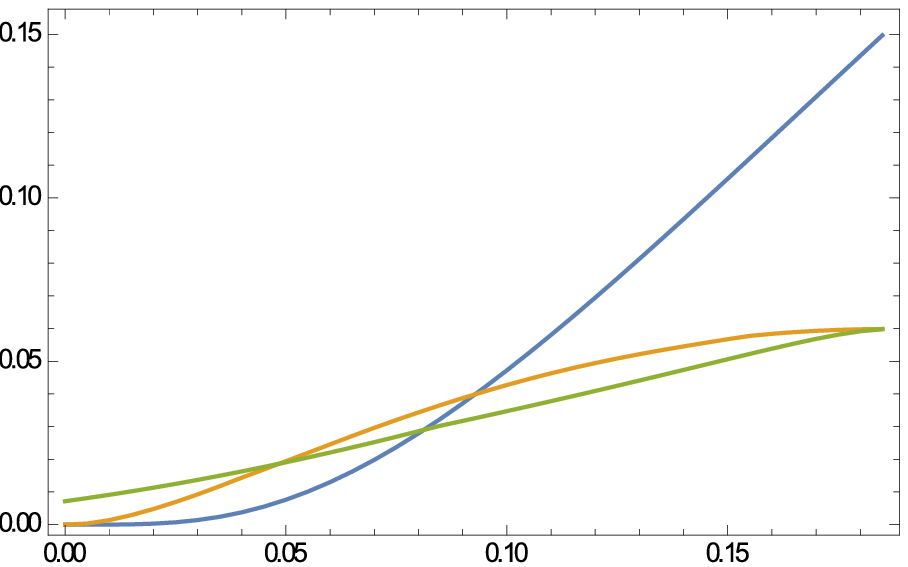}}
   \subfloat[][$t=0.70\times20$]{\includegraphics[scale=0.45]{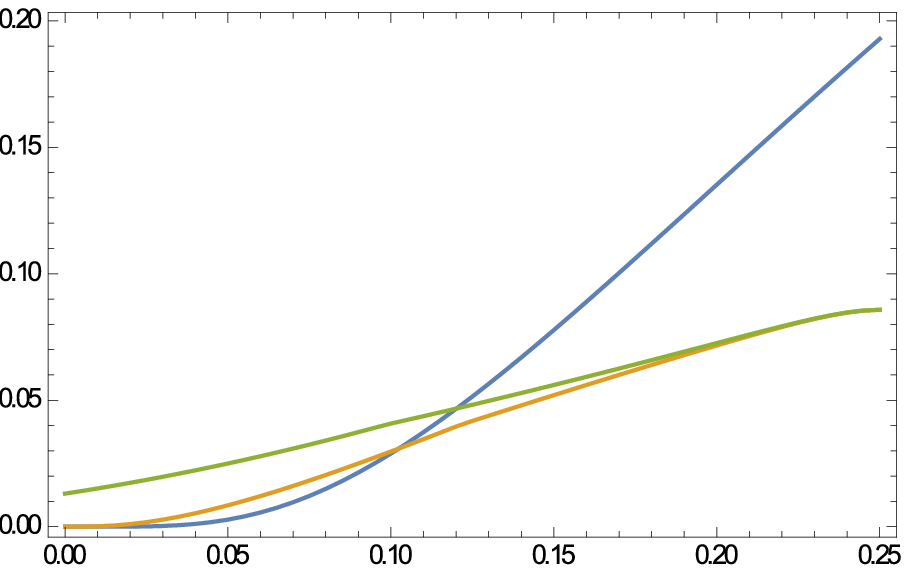}}
   \subfloat[][$t=0.95\times20$]{\includegraphics[scale=0.45]{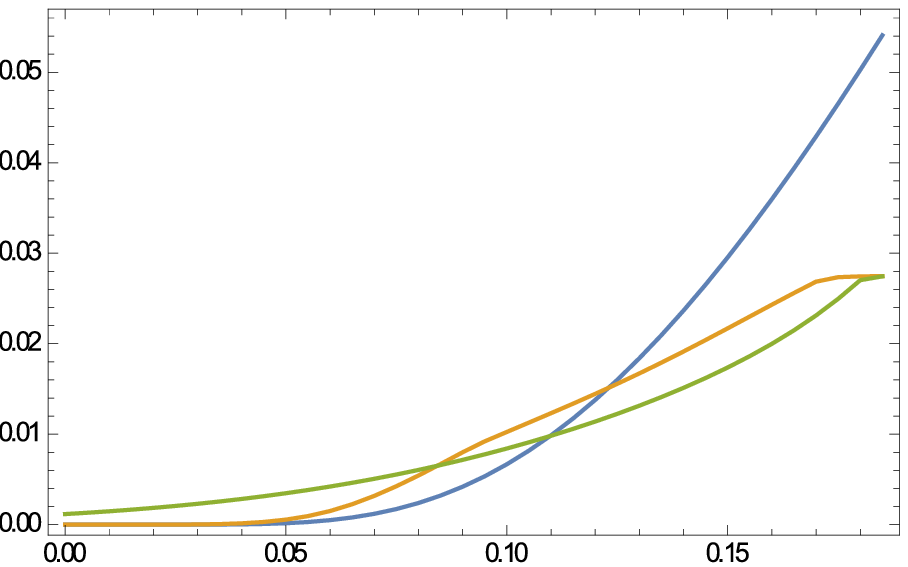}}
   \caption{Comparison of Bennett's bound and the bounds of Theorems \ref{momopt} and \ref{xitheorem} 
   for $X_i\in\mathcal{B}(p,\sigma^2)$. The first column corresponds to $p=1/4$. The second column 
   corresponds to $p=1/2$. Finally, the third column corresponds to the value $p=3/4$.
   We set $n=20$ and $t=(p+\epsilon)n$, for particular choices of $\epsilon$. 
   The blue curves represent Bennett's bound; the orange curves represent the bound given 
   by Theorem \ref{xitheorem}; 
   the green curves correspond to the bound given by Theorem \ref{momopt}.
   The abscissae are the variances.}\label{fig:compare}
\end{figure}

Our numerical experiments suggest that, when $\sigma^2$ is not very small, 
the bound given by Theorem \ref{xitheorem} is tighter than 
Bennett's bound.
Note that we can also apply the bound given by Theorem \ref{momopt}  to random variables 
from the class $\mathcal{B}(p,\sigma^2)$; it is not difficult to implement this bound.
In order to build a concrete mental image let us fix the parameter $p$ and consider
random variables $X_i, i=1,\ldots,n$ such that $X_i \in \mathcal{B}(0.5,\sigma^2).$
In a similar way as in  Proposition \ref{Yuyiprtwo} one can show that it suffices to 
consider the infimum, in the bound of Theorem \ref{xitheorem}, over the set 
$K:=\{k/2 : k\; \text{is nonnegative integer and}\; k/2 <t\}$. 
We can now put the computer to work to calculate the bound  
\[ \inf_{\epsilon\in K} \mathbb{E}\left[\max\left\{0,\;\frac{\sum_{i=1}^n \xi_{p_i,\sigma_i}-\epsilon}{t-\epsilon}\right\}\right]. \]
Figure \ref{fig:compare} shows comparisons between  Bennett's bound, the bound obtained in Theorem \ref{momopt}  
and the bound of Theorem \ref{xitheorem}. The abscissae in these figures correspond to the variance. 
Notice that, when the variance is large, 
the bounds given by Theorems \ref{momopt} and \ref{xitheorem} are sharper than the Bennett bound.

In the next section we stretch a limitation of the Bernstein-Hoeffding method.

\section{Unbounded random variables}\label{unbounded}

So far we have employed the Bernstein-Hoeffding method to sums of independent and \emph{bounded} random variables. 
The reader may wonder whether the method can be employed in order to obtain 
bounds on deviations from the expectation for sums of independent, non-negative and \emph{unbounded} random variables. 
We will show, in this section, that in this case the method yields a bound that is the same as the bound 
given by Markov's inequality. Let us remark that this fact was already known to Hoeffding (see the footnote 
in \cite{Hoeffdingone}, page $15$) but we were not able to find a proof; we include a proof for the sake of completeness. 
Hence the case of non-negative and unbounded random variables 
requires different methods and the reader is invited to take a 
look at the work of Samuels \cite{Samuelsone}, \cite{Samuelstwo}, \cite{Samuelsthree} and 
Feige \cite{Feige} for further details and references. The case of non-negative and 
unbounded random variables seems to be less investigated than the case of bounded random variables. 
Talagrand (see \cite{Talagrand}, page 692, Comment $3$) already mentions that it is unclear how to improve 
Hoeffding's inequality without the assumption that the random variables are bounded from above.
Let us fix some notation. \\

For given $\mu>0$, let $\mathcal{U}(\mu)$ the class of non-negative random variables whose mean equals $\mu$. 
Formally, 
\[  \mathcal{U}(\mu) = \{X: X\geq 0,\; \mathbb{E}[X]=\mu\} . \]

Now, for $i=1,\ldots,n$, fix $\mu_i \geq 0$ and  $X_i \in \mathcal{U}(\mu_i)$. If $t>\sum_i \mu_i$, then one can estimate
\[ \mathbb{P}\left[\sum_{i=1}^{n}X_i \geq t\right] \leq \frac{1}{f(t)} \mathbb{E}\left[ f\left(\sum_{i=1}^{n}X_i\right)\right], \]
where $f(\cdot)$ is a non-negative, convex and increasing function. A crucial step in 
the Bernstein-Hoeffding method is to 
minimise the right hand side of the last inequality with respect to $f(\cdot)$.
We may assume that we minimise over those functions $f$ for which $f(t)=1$. 
We now show that this minimisation leads to a bound that is the same as Markov's. Note that Markov's inequality yields
$\mathbb{P}\left[\sum_{i=1}^{n}X_i \geq t\right]\leq \frac{\sum_i \mu_i}{t}$.
Recall the definition of the class $\mathcal{F}_{ic}(t)$, from the Introduction, and let 
$\mathcal{Z}_{ic}(t)$ be the class consisting of all functions $f\in \mathcal{F}_{ic}(t)$ such that $f(t)=1$. 
In this section we report the following. \\

\begin{prop}\label{unbound} With the same notation as above, we have
\[ V:= \inf_{f\in \mathcal{Z}_{ic}(t)} \sup_{X_i\in \mathcal{U}(\mu_i)} \mathbb{E}\left[f\left(\sum_{i=1}^{n}X_i\right)\right] \geq \frac{\sum_{i=1}^{n}\mu_i}{t}.\] 
\end{prop}
\begin{proof} For $i=1,\ldots,n$, let $Y_i$ be the random variable that 
takes the values $0$ and $t$ with probabilities $1-\frac{\mu_i}{t}$ and $\frac{\mu_i}{t}$, respectively. 
Clearly, we have 
\[ V \geq \inf_{f\in \mathcal{Z}_{ic}(t)} \mathbb{E}\left[ f\left(\sum_{i=1}^{n}Y_i\right)\right] . \]
In a similar way as in Theorem \ref{Yuyipr} one can show that 
\[ \inf_{f\in \mathcal{Z}_{ic}(t)} \mathbb{E}\left[ f\left(\sum_{i=1}^{n}Y_i\right)\right] = \inf_{\varepsilon \in [0,t)} \mathbb{E}\left[\max\left\{0, \frac{\sum_i Y_i -\varepsilon}{t-\varepsilon}\right\}\right]. \]
Since $\sum_i Y_i \in \{0,t,2t,\ldots,nt\}$, a similar argument as in Proposition \ref{Yuyiprtwo} shows that the optimal 
$\varepsilon$ in the right hand side of the last equation is equal to $0$. Therefore, 
\[ \inf_{\varepsilon \in [0,t)} \mathbb{E}\left[\max\left\{0, \frac{\sum_{i=1}^{n} Y_i -\varepsilon}{t-\varepsilon}\right\}\right]= 
\frac{1}{t} \mathbb{E}\left[ \sum_{i=1}^{n} Y_i \right] = \frac{1}{t} \sum_{i=1}^{n} \mu_i \]
and the result follows. 
\end{proof}

Hence, in the case of non-negative and unbounded random variables,
the  method cannot yield a bound that is better than Markov's bound.\\

\textbf{Acknowledgements}
The authors are supported by ERC Starting Grant 240186 "MiGraNT, Mining Graphs and Networks: a Theory-based approach". 
We are grateful to Xiequan Fan for several valuable suggestions and comments.

\end{document}